\documentclass[printscheme,screen]{aomart}

\usepackage{amssymb}
\usepackage{xurl}
\usepackage{booktabs}
\usepackage{tabularx}
\usepackage{array}
\usepackage{longtable}
\usepackage{makecell}
\usepackage[shortlabels]{enumitem}
\copyrightnote{}

\numberwithin{equation}{section}
\theoremstyle{plain}
\newtheorem{theorem}{Theorem}[section]
\newtheorem{proposition}[theorem]{Proposition}
\newtheorem{lemma}[theorem]{Lemma}
\newtheorem{corollary}[theorem]{Corollary}
\theoremstyle{definition}

\theoremstyle{remark}
\newtheorem{remark}[theorem]{Remark}

\renewcommand{\theequation}{%
  \ifnum\value{section}=1
    \arabic{equation}%
  \else
    \thesection.\arabic{equation}%
  \fi}
\renewcommand{\thetheorem}{%
  \ifnum\value{section}=1
    \arabic{theorem}%
  \else
    \thesection.\arabic{theorem}%
  \fi}
\title[Hilbert eigenform product identities]
  {Hecke Eigenform Product Identities for Full-Level Hilbert Modular Forms}

\author[D. Chen]{Dijia Chen}
\givenname{Dijia}
\surname{Chen}
\address{Department of Mathematics, University of Wisconsin--Madison,
  Madison, WI, USA}
\email{dchen426@wisc.edu}

\author[B. Gan]{Bowen Gan}
\givenname{Bowen}
\surname{Gan}
\address{Institute of Mathematical Sciences, ShanghaiTech University,
  Shanghai, China}
\email{ganbw2023@shanghaitech.edu.cn}

\author[C. Qin]{Chao Qin}
\givenname{Chao}
\surname{Qin}
\address{College of Mathematical Sciences, Harbin Engineering University,
  Harbin, China}
\email{qinchao@hrbeu.edu.cn}

\author[Y. Zhou]{Yang Zhou}
\givenname{Yang}
\surname{Zhou}
\address{College of Mathematical Sciences, Harbin Engineering University,
  Harbin, China}
\email{yang.zhou@hrbeu.edu.cn}

\renewcommand{\shortauthors}{D. CHEN, B. GAN, C. QIN, and Y. ZHOU}
\corresponding{Yang Zhou}

\makeatletter
\let\annals@firstpagestyle\ps@firstpage
\def\ps@firstpage{%
  \annals@firstpagestyle
  \fancyfoot{}%
}
\patchcmd{\@setaddresses}
  {\goodbreak\addvspace\bigskipamount}
  {\goodbreak\addvspace\medskipamount}{}{}
\makeatother


\begin{document}

\begin{abstract}
We study full-level Hilbert Hecke eigenform product identities over totally real fields of degree at least two, allowing arbitrary compatible narrow class characters. Combining our
results with previous work, we classify products of two Eisenstein eigenforms of positive parallel integral weights and products of an Eisenstein eigenform of weight at least two with a cuspidal eigenform of positive parallel integral weight. The unique identity in each of these two classes is given by $E_{4}=60 E_{2}^{2}$ and $h_{8}=120 E_{2} \cdot h_{6}$, respectively, both over $\mathbb{Q}(\sqrt{5})$. The proofs use constant-term comparisons, indecomposable Fourier exponents, integrality of reciprocal Eisenstein constant terms, and a Rankin-Selberg formula for Petersson inner products. In the remaining case of a weight-one Eisenstein factor and a cuspidal factor, we also establish
\(h_5=12E_1(1,\varepsilon)\cdot h_4\) over
\(\mathbb Q(\sqrt3)\), where \(\varepsilon\) is the nontrivial
narrow class character. All eigenforms in these identities have unit-ideal coefficient one.
\end{abstract}

\maketitle
\thispagestyle{empty}
\raggedbottom
\tableofcontents

\section{Introduction and statements of the main theorems}
When can the product of two Hecke eigenforms be a Hecke eigenform? For
\(\mathrm{SL}_2(\mathbb Z)\), Duke \cite{Duke1997product} and
Ghate \cite{Ghate2000Eisenstein} independently obtained 16 identities,
each arising from a one-dimensional space of modular forms. Johnson
\cite{Johnson2013eigenforms} extended the classification to arbitrary
level and Nebentypus for weights at least two, obtaining
61 identities.

In the Hilbert modular setting, Joshi--Zhang
\cite{Joshi-Zhang2019Hilbert} proved the finiteness of full-level
eigenform product identities among Hecke eigenforms of parallel weights
at least two over real quadratic fields, and exhibited two identities
over \(F=\mathbb Q(\sqrt5)\). You--Zhang
\cite{You-Zhang2021Hilbert} subsequently established a corresponding
finiteness theorem over totally real fields of any fixed degree.
Hao--Qin--Zhou \cite[Theorems~1 and~2]{HaoQinZhou2026} classified
products of two Eisenstein eigenforms over real quadratic fields with
narrow class number one and excluded distinct-weight Eisenstein products
in degree at least three. They also classified Eisenstein–
cuspidal products in the quadratic setting, using the Generalized Riemann
Hypothesis (GRH) when the Eisenstein weight is two.
Li \cite[Theorem~3.9]{Li2026RankinCohen} removed this assumption
using a Petersson--Rankin--Selberg identity for Hilbert Rankin--Cohen
brackets.

In this paper, we study eigenform product identities among full-level
Hilbert Hecke eigenforms of positive parallel integral weights over
totally real fields of degree at least two, allowing arbitrary
compatible narrow class characters.
An eigenform product identity is an equality
\(g=\lambda f\cdot h\), where \(f,h,g\) are nonzero Hecke eigenforms
and \(\lambda\in\mathbb C^\times\). We normalize each eigenform by
its unit-ideal coefficient and classify identities up to
interchanging the factors. The displayed scalars account for
normalization of the product.

Our first result classifies all products of two Eisenstein eigenforms.
\begin{theorem}
\label{thm:main-ee}
Let \(F\) be a totally real field of degree \(n\geq2\).
The only eigenform product identity for two full-level Eisenstein
Hecke eigenforms of parallel weights \(k_1,k_2\geq1\) is
\[
E_4=60E_2^2,
\]
over \(F=\mathbb Q(\sqrt5)\). Here \(E_k\) denotes the normalized
Eisenstein eigenform of weight \(k\). This identity was established in
\cite[Theorem~7.4]{Joshi-Zhang2019Hilbert}.
\end{theorem}

The second theorem treats an Eisenstein factor of weight at least two
and a cuspidal factor of any positive weight.

\begin{theorem}
\label{thm:main-ec}
Let \(F\) be a totally real field of degree \(n\geq2\).
Among products of a full-level Eisenstein Hecke eigenform of
parallel weight \(k\geq2\) and a full-level cuspidal Hecke eigenform
of parallel weight \(\ell\geq1\), the only eigenform product identity is
\[
h_8=120E_2\cdot h_6,
\]
over \(F=\mathbb Q(\sqrt5)\).
Here \(E_2\) is the normalized Eisenstein eigenform of weight two,
and \(h_6,h_8\) are the unique normalized cuspidal Hecke eigenforms
of weights six and eight, respectively. This identity was established
in \cite[Theorem~7.4]{Joshi-Zhang2019Hilbert}.
\end{theorem}

We also prove a product identity with a weight-one Eisenstein factor over \(\mathbb Q(\sqrt3)\).

\begin{theorem}
\label{thm:weight-one-identity}
Let \(F=\mathbb Q(\sqrt3)\), and let \(\varepsilon\) be the
nontrivial narrow class character of \(F\). Then the full adelic cuspidal spaces satisfy
\[
\dim S_4(1)=2,\qquad \dim S_5(\varepsilon)=1.
\]
For either normalized full-level cuspidal Hecke eigenform
\(h_4\in S_4(1)\), one has
\begin{equation*}
h_5=12E_1(1,\varepsilon)\cdot h_4,
\end{equation*}
where \(h_5\) is the unique normalized Hecke eigenform in
\(S_5(\varepsilon)\). In particular, the two choices of \(h_4\)
give the same product.
\end{theorem}

Table~\ref{tab:full-level-cases} summarizes the classification
according to the types and weights of the two factors.
We group together cases that use the same main method of proof,
although the detailed calculations may differ within each group.

\begingroup
\small
\setlength{\tabcolsep}{5pt}
\setlength{\LTcapwidth}{\textwidth}
\renewcommand{\arraystretch}{1.10}

\begin{longtable}{@{}
>{\raggedright\arraybackslash}p{0.18\textwidth}
>{\raggedright\arraybackslash}p{0.44\textwidth}
>{\raggedright\arraybackslash}p{\dimexpr0.38\textwidth-20pt\relax}
@{}}
\caption[Full-level eigenform products: results and methods.]{%
\centering
Full-level eigenform products: results and methods.\endgraf
\raggedright
Here \(E_k\) and \(C_\ell\) denote Eisenstein and cuspidal Hecke eigenforms
of weights \(k\) and \(\ell\);\newline
\(n=[F:\mathbb Q]\geq2\).
}
\label{tab:full-level-cases}\\

\toprule
\textbf{Range} & \textbf{Conclusion} & \textbf{Method}\\
\midrule
\endfirsthead

\multicolumn{3}{@{}l}{\tablename\ \thetable\ (continued)}\\
\toprule
\textbf{Range} & \textbf{Conclusion} & \textbf{Method}\\
\midrule
\endhead

\bottomrule
\endfoot

\(C_{\ell_1}\times C_{\ell_2}\)\newline
\(\ell_i\geq1\) &
No identity &
Unit-ideal coefficient; see Section~\ref{sec:cuspidal-products}.
\\
\addlinespace[5pt]

\(E_{k_1}\times E_{k_2}\)\newline
\(n=2,\ k_i\geq2\) &
Only \(E_4=60E_2^2\)\newline
over \(\mathbb Q(\sqrt5)\) &
Known classification \cite[Theorem 1(1)]{HaoQinZhou2026},
positivity and discriminant bounds; see
Section~\ref{sec:quadratic-ee-products}.
\\
\addlinespace[5pt]

\(E_{k_1}\times E_{k_2}\)\newline
\(n>2,\ k_i\geq2\) &
No identity &
Special-value, coefficient and class-number bounds; see
Section~\ref{sec:higher-degree-ee-products}.
\\
\addlinespace[5pt]

\(E_1\times E_k\)\newline
\(k\geq1\) &
No identity &
Constant terms on narrow components; see
Section~\ref{sec:weight-one-eisenstein-products}.
\\
\addlinespace[5pt]

\(E_k\times C_\ell\)\newline
\(k\geq3,\ \ell\geq1\) &
No identity &
Integrality, discriminant bounds and exact special values; see
Section~\ref{sec:eisenstein-cuspidal-weight-at-least-three}.
\\
\addlinespace[5pt]

\(E_2\times C_\ell\)\newline
\(\ell\geq1\) &
Only \(h_8=120E_2\cdot h_6\)\newline
over \(\mathbb Q(\sqrt5)\) &
Petersson pairing and dimension criterion; see
Section~\ref{sec:eisenstein-cuspidal-weight-two-products}.
\\
\addlinespace[5pt]

\(E_1\times C_\ell\)\newline
\(\ell\geq1\) &
Open in general.\newline
\(h_5=12E_1(1,\varepsilon)\cdot h_4\)\newline
over \(\mathbb Q(\sqrt3)\) &
Cuspidal dimensions for the example; see
Section~\ref{sec:weight-one-products}.
\\

\end{longtable}
\endgroup

Section~\ref{sec:preliminaries} fixes our notation and records the
formulas and background needed in the proofs.
Section~\ref{sec:cuspidal-products} treats products of two cuspidal eigenforms.
Sections~\ref{sec:quadratic-ee-products}--\ref{sec:weight-one-eisenstein-products}
prove Theorem~\ref{thm:main-ee}.
Sections~\ref{sec:eisenstein-cuspidal-weight-at-least-three}--\ref{sec:eisenstein-cuspidal-weight-two-products}
prove Theorem~\ref{thm:main-ec}, treating Eisenstein weights at least
three and weight two, respectively.
Section~\ref{sec:weight-one-products} proves the weight-one Eisenstein
identity in Theorem~\ref{thm:weight-one-identity}.

\section{Preliminaries}
\label{sec:preliminaries}

We first fix the field and character notation and the Fourier
coefficient normalization. We then record special-value bounds and product
coefficient identities, followed by the Petersson and Rankin--Selberg
conventions used for Eisenstein weight two. The final subsection
explains the finite calculations used in the proofs.
All weights are parallel, and all ideals in sums are nonzero integral
ideals unless otherwise stated.

\subsection{Characters and arithmetic notation}

We record the parity conditions on class characters and the
arithmetic bounds used to restrict the possible fields.
Let $F$ be a totally real field of degree $n\geq2$, with ring of
integers $\mathcal O_F$, different $\mathfrak d_F$, discriminant $D_F$,
ordinary class number $h(F)$, and narrow class number $h^+(F)$.
Put $\delta_F=D_F^{1/n}$.
We write $N\mathfrak a$ for the absolute norm of a nonzero fractional
ideal $\mathfrak a$, and $N(a)=N_{F/\mathbb Q}(a)$ for the field norm
of $a\in F$. Hence $N((a))=|N(a)|$ for $a\in F^\times$.
We write $x\gg0$ when $x$ is totally positive.

A narrow class character is identified with a finite-order Hecke
character unramified at every finite place. If its real components
are $\chi_v(x)=\operatorname{sgn}(x)^{r_v}$, then for $a\in F^\times$ we have
\begin{equation}
\label{eq:prelim-principal-character}
\chi((a))=\prod_{v\mid\infty}\operatorname{sgn}(a_v)^{r_v}.
\end{equation}
We call $\chi$ totally even if all $r_v=0$, and totally odd if all
$r_v=1$. Totally even characters are precisely the ordinary class
characters. For a totally odd character, applying
\eqref{eq:prelim-principal-character} to units shows that every unit
of $F$ has norm one.

We use the following bounds of Minkowski
\cite[\S5]{Lenstra1992Algorithms} and Odlyzko--Takeuchi
\cite[Proposition~2.3]{TAKEUCHI83}:
\begin{equation}
\label{eq:prelim-arithmetic-bounds}
\min_{\mathfrak a\in\mathcal C}N\mathfrak a
\leq\frac{n!}{n^n}\sqrt{D_F},\qquad
\sqrt{D_F}\geq\frac{n^n}{n!},\qquad
\delta_F>29.099e^{-8.3185/n},
\end{equation}
where the minimum is over integral ideals in an ordinary ideal
class $\mathcal C$. In degrees three and four, respectively,
$D_F\geq49$ and $D_F\geq725$ \cite[Table~3]{Voight}.

Let $H$ be the Hilbert class field corresponding
to the ordinary ideal class group of $F$.
The extension $H/F$ is unramified at every finite prime, and every
real place of $F$ splits completely in $H$.
It follows that $H$ is totally real, $[H:\mathbb Q]=nh(F)$, and
$\delta_H=\delta_F$.

\subsection{Hilbert modular forms and Hecke coefficients}

We fix the Fourier coefficient normalization and recall the Hecke
relation used to compare products, following \cite[\S2.1]{Dasgupta2011Hilbert}. We denote the space of full-level holomorphic  adelic Hilbert modular forms by $M_k$. Let $\chi$ be a narrow ideal class character. The subspace with central character $\chi$ is denoted by $M_k(\chi)$.
Choose representatives of the narrow ideal classes
$t_1=\mathcal O_F,t_2,\ldots,t_{h^+(F)}$.
A full-level form $f=(f_\nu)$ of weight $k$ has components for
$\Gamma_0(t_\nu\mathfrak d_F,\mathcal O_F)$, where the group is defined by
\[
\Gamma_0(\mathfrak c,\mathcal O_F)
=\left\{
\begin{pmatrix}a&b\\c&d\end{pmatrix}:
 a,d\in\mathcal O_F,\ b\in\mathfrak c^{-1},\ c\in\mathfrak c,
\ ad-bc\in\mathcal O_F^{\times,+}
\right\}.
\]
Write
\[
f_\nu(z)=a_\nu(0)+\sum_{\substack{\mu\in t_\nu\\\mu\gg0}}
a_\nu(\mu)e^{2\pi i\operatorname{Tr}(\mu z)}.
\]
For $\mathfrak m=(\mu)t_\nu^{-1}$, our coefficient normalization is
\begin{equation*}
c(\mathfrak m,f)=N(t_\nu)^{-k/2}a_\nu(\mu),\qquad
c_\nu(0,f)=N(t_\nu)^{-k/2}a_\nu(0).
\end{equation*}
Put
$K_f=\mathrm{GL}_2(\widehat{\mathcal O}_F)$, and write
\[
M_k(\chi)=S_k(\chi)\oplus\mathcal E_k(\chi),
\]
where $S_k(\chi)$ and $\mathcal{E}_k(\chi)$ are the subspaces of cusp forms and Eisenstein series, respectively.
Cuspidal forms vanish at every cusp, and hence have
$c_\nu(0,f)=0$ for every $\nu$.
The space $M_k(\chi)$ vanishes unless
$\chi_v(-1)=(-1)^k$ at every real place.

\begin{lemma}
\label{lem:odd-weight-unit-parity}
If $F$ admits a nonzero full-level form of odd parallel weight,
then $n$ is even and every unit of $F$ has norm one.
In particular, no such form exists when $h^+(F)=1$.
\end{lemma}
\begin{proof}
The central character is totally odd. Applying
\eqref{eq:prelim-principal-character} to units gives the norm assertion,
and applying it to $-1$ shows that $n$ is even. If $h^+(F)=1$, the
only narrow class character is trivial, which is incompatible
with odd weight.
\end{proof}

Hecke eigenforms are normalized by $c(\mathcal O_F,f)=1$.
For a prime ideal $\mathfrak p$, we have
\begin{equation}
\label{eq:prelim-prime-power-recurrence}
c(\mathfrak p^{r+1},f)
=c(\mathfrak p,f)c(\mathfrak p^r,f)
-\chi(\mathfrak p)N(\mathfrak p)^{k-1}c(\mathfrak p^{r-1},f)
\quad(r\geq1).
\end{equation}
In particular, we have
\begin{equation}
\label{eq:hecke-prime-square-recurrence}
c(\mathfrak p^2,f)=c(\mathfrak p,f)^2
-\chi(\mathfrak p)N(\mathfrak p)^{k-1}.
\end{equation}
For $k\geq2$, these coefficients are algebraic integers
\cite[Proposition~2.2]{Shimura1978Hilbert}.
\begin{lemma}
\label{lem:weight-one-coefficients}
Let $f\in S_k(\chi)$ be a normalized full-level cuspidal eigenform
of parallel weight $k\geq1$.
Its ideal coefficients are algebraic integers and satisfy
\begin{equation}
\label{eq:prelim-ramanujan-bound}
|\sigma(c(\mathfrak a,f))|
\leq\tau_F(\mathfrak a)N(\mathfrak a)^{(k-1)/2}
\end{equation}
for every nonzero integral ideal $\mathfrak a$ and every embedding
$\sigma$ of its coefficient field, where $\tau_F(\mathfrak a)$
denotes the number of integral ideal divisors of $\mathfrak a$.
In particular, when $k=1$,
\[
|\sigma(c(\mathfrak p^r,f))|\leq r+1
\qquad(r\geq0).
\]
\end{lemma}

\begin{proof}
For $k\geq2$, algebraic integrality follows from
\cite[Proposition~2.2]{Shimura1978Hilbert}.
Fix a prime ideal $\mathfrak p$, put $q=N\mathfrak p$, and let
$A_{\mathfrak p},B_{\mathfrak p}$ be the roots of
\[
X^2-c(\mathfrak p,f)X+\chi(\mathfrak p)q^{k-1}.
\]
The Hecke relation \eqref{eq:prelim-prime-power-recurrence} gives
\[
c(\mathfrak p^r,f)
=\sum_{j=0}^r A_{\mathfrak p}^{r-j}B_{\mathfrak p}^j.
\]
Extend $\sigma$ to an embedding $\widetilde\sigma$ of a number
field containing both roots.
For $k\geq2$, Blasius's Ramanujan theorem
\cite[Theorem~1]{Blasius2006Ramanujan}, applied to the Galois
conjugates of $f$, gives
\[
|\widetilde\sigma(A_{\mathfrak p})|
=|\widetilde\sigma(B_{\mathfrak p})|
=q^{(k-1)/2}.
\]

For $k=1$, let $\rho_f$ be the finite-image Artin representation
attached to $f$; see \cite[\S1]{BetinaDeoFite2020}.
Since $f$ has full level, $\rho_f$ is unramified at every finite
prime, and the polynomial above is its arithmetic Frobenius
polynomial at $\mathfrak p$.
Thus both roots are roots of unity, as are their images under
$\widetilde\sigma$, so the displayed equality also holds in
weight one. The expansion additionally shows that
$c(\mathfrak p^r,f)$ is an algebraic integer in this case.

Consequently, for every $k\geq1$,
\[
|\sigma(c(\mathfrak p^r,f))|
\leq(r+1)q^{r(k-1)/2}.
\]
Multiplicativity and
$\tau_F(\mathfrak a)=\prod_{\mathfrak p^r\Vert\mathfrak a}(r+1)$
give \eqref{eq:prelim-ramanujan-bound} for every nonzero integral
ideal, and also establish algebraic integrality in weight one.
\end{proof}

For a narrow class character $\xi$, we use the Hecke-normalized
twist, characterized by
\begin{equation}
\label{eq:prelim-character-twist}
f\otimes\xi\in M_k(\chi\xi^2),\qquad
c(\mathfrak a,f\otimes\xi)=\xi(\mathfrak a)c(\mathfrak a,f).
\end{equation}
In an adelic realization, this is multiplication by
$\xi(\det g)$ up to a fixed nonzero scalar determined by our
Fourier normalization.
Twisting preserves full level and cuspidality and is invertible.
It applies to arbitrary forms, including products that are not eigenforms.

\subsection{Special values and Eisenstein series}

We next recall the special-value formulas and bounds needed for
Eisenstein series.
For a narrow class character $\chi$, we write
\[
L_F(s,\chi)=L(s,\chi)
=\prod_{\mathfrak p}(1-\chi(\mathfrak p)N(\mathfrak p)^{-s})^{-1}
\qquad(\operatorname{Re}(s)>1).
\]
Rationality of partial zeta values \cite[\S7.1]{GunnellsSczech2003}
implies
\begin{equation}
\label{eq:hecke-L-galois-equivariance}
\sigma(L(1-m,\chi))=L(1-m,\chi^\sigma)
\qquad(m\geq1).
\end{equation}
Here $\sigma$ acts on a cyclotomic field containing the character
values. It preserves their archimedean parity.
For $m\geq2$ with $\chi_v(-1)=(-1)^m$, the functional equation
for a character of finite conductor one
\cite[\S2.4]{MolinPage2022Hecke} gives
\begin{equation}
\label{eq:prelim-hecke-functional-equation}
|L(1-m,\chi)|
=\frac{2^nD_F^{m-1/2}\Gamma(m)^n}{(2\pi)^{nm}}
 |L(m,\chi^{-1})|.
\end{equation}
The discriminant and gamma factors in this equation are independent
of the character. Write their product as
$C_F(m)=2^nD_F^{m-1/2}\Gamma(m)^n/(2\pi)^{nm}$.
Bounds for the remaining $L$-value at $s=m$ give the following
estimate \cite[equation~(2.4)]{You-Zhang2021Hilbert}.
\begin{equation}
\label{eq:hecke-L-value-bounds}
C_F(m)\frac{\zeta(n^2m)}{\zeta(m)^n}
\leq |L(1-m,\chi)|\leq C_F(m)\zeta(m)^n.
\end{equation}
In particular, the special value $L(1-m,\chi)$ is nonzero.
By \eqref{eq:hecke-L-galois-equivariance}, these bounds also hold
for every Galois conjugate of $L(1-m,\chi)$.
The functional equation also gives $(-1)^n\zeta_F(-1)>0$.
For a totally even character $\chi$,
\eqref{eq:prelim-hecke-functional-equation} at $m=2$ reads
\[
|L(-1,\chi)|=\frac{D_F^{3/2}}{(2\pi^2)^n}|L(2,\chi^{-1})|.
\]

To compute Dirichlet special values, let $\xi$ be a primitive
character of conductor $q$. For every integer $m\geq1$,
the generalized Bernoulli formula \cite[Chapter~4]{Washington1997Cyclotomic} reads
\begin{equation}
\label{eq:generalized-bernoulli-special-value}
L(1-m,\xi)=-\frac{B_{m,\xi}}m,\qquad
B_{m,\xi}=q^{m-1}\sum_{a=1}^{q}
\xi(a)B_m(a/q).
\end{equation}
Here $B_m(x)$ is the $m$-th Bernoulli polynomial.
In the finite reductions below, we need explicit intervals containing
$\zeta_F(2)$. Truncating its Euler product gives the rational lower bound
\[
P_B(F)=\prod_{p\leq B}\prod_{\mathfrak p\mid p}
(1-N(\mathfrak p)^{-2})^{-1}.
\]
For every integer $B\geq2$, we have
\begin{equation}
\label{eq:prelim-euler-product-bound}
P_B(F)<\zeta_F(2)
<P_B(F)\left(\frac{B+1}{B}\right)^n.
\end{equation}
To obtain the upper bound, we bound the Euler factor at each
rational prime $p>B$ by $(1-p^{-2})^{-n}$ and use
\[
\prod_{p>B}(1-p^{-2})^{-1}
<
\prod_{m=B+1}^{\infty}(1-m^{-2})^{-1}
=\frac{B+1}{B}.
\]

For $k\geq2$ and narrow class characters $\phi,\psi$ with
$(\phi\psi)_v(-1)=(-1)^k$, the normalized Eisenstein eigenform
$E_k(\phi,\psi)\in\mathcal E_k(\phi\psi)$ satisfies
\begin{equation}
\label{eq:eisenstein-fourier-and-constant-terms}
\begin{aligned}
c(\mathfrak m,E_k(\phi,\psi))
 &=\sum_{\mathfrak r\mid\mathfrak m}
 \phi(\mathfrak m\mathfrak r^{-1})\psi(\mathfrak r)
 N(\mathfrak r)^{k-1},\\
c_\nu(0,E_k(\phi,\psi))
 &=2^{-n}\phi^{-1}(t_\nu)L(1-k,\phi^{-1}\psi).
\end{aligned}
\end{equation}
See \cite[Proposition~3.4]{Shimura1978Hilbert} and
\cite[Proposition~2.1]{Dasgupta2011Hilbert}.
The quotient character $\phi^{-1}\psi$ has the parity required
in \eqref{eq:hecke-L-value-bounds}.
For $k\geq2$, these series form a Hecke eigenbasis of the
full-level Eisenstein space, indexed by ordered inducing pairs;
see \cite[\S4]{DasguptaKakde2021}.

In weight one, the nonconstant coefficient formula remains valid,
while the corresponding two-term constant formula is given in
Section~\ref{sec:weight-one-eisenstein-products}.
Since $E_1(\phi,\psi)=E_1(\psi,\phi)$, a Hecke eigenbasis is
indexed by unordered inducing pairs
\cite[\S4]{DasguptaKakde2021}.

\subsection{Product coefficient identities}
\label{subsec:common-product-coefficients}

We derive the coefficient identities at small ideals that constrain
an eigenform product. Multiplication is componentwise, and
$M_{k_1}(\chi_1)\cdot M_{k_2}(\chi_2)\subseteq M_{k_1+k_2}(\chi_1\chi_2)$.
The product is cuspidal if either factor is cuspidal.
For $\mathfrak a=(\mu)t_\nu^{-1}$, $\mu\gg0$, multiplication of
Fourier expansions gives
\begin{equation}
\label{eq:component-product-convolution}
\begin{aligned}
c(\mathfrak a,f\cdot h)
={}&c_\nu(0,f)c(\mathfrak a,h)+c_\nu(0,h)c(\mathfrak a,f)\\
&+\sum_{\substack{x\in t_\nu\\0\ll x\ll\mu}}
 c((x)t_\nu^{-1},f)c((\mu-x)t_\nu^{-1},h),
\end{aligned}
\end{equation}
and $c_\nu(0,f\cdot h)=c_\nu(0,f)c_\nu(0,h)$.
An exponent is called indecomposable if the sum is empty.
If $\mu=x+y$ with $x,y\in t_\nu$ totally positive, apply
Minkowski's determinant inequality \cite[Theorem~2]{MarcusLopes1957}
to their positive embedding vectors.
Dividing by $N(t_\nu)^{1/n}$ gives
\begin{equation}
\label{eq:indecomposable-norm-bound}
N(\mathfrak a)^{1/n}
\geq N((x)t_\nu^{-1})^{1/n}+N((y)t_\nu^{-1})^{1/n}\geq2.
\end{equation}
If equality holds throughout \eqref{eq:indecomposable-norm-bound},
then the integral ideals $(x)t_\nu^{-1}$ and $(y)t_\nu^{-1}$
both have norm one and hence equal $\mathcal O_F$.
Moreover, equality in Minkowski's determinant inequality implies
that the positive embedding vectors of $x$ and $y$ are proportional.
Since the two ideals are equal, $N(x)=N(y)$,
so the proportionality factor is one.
It follows that $x=y$ and $\mathfrak a=(2)$.
Consequently, exponents corresponding to ideals of norm below
$2^n$ are indecomposable, as are those corresponding to ideals
of norm $2^n$ other than $(2)$.

In the principal component, the exponent $1$ is therefore
indecomposable.
For the exponent $2$, the equality case forces $x=y=1$,
so its only ordered positive decomposition is $2=1+1$.
These observations yield the following coefficient identities
for the unit ideal and $(2)$.

\begin{lemma}
\label{lem:low-product-coefficients}
For full-level forms $f\in M_{k_1}(\chi_1)$ and
$h\in M_{k_2}(\chi_2)$, we have
\begin{gather}
c(\mathcal O_F,f\cdot h)
=c_1(0,f)c(\mathcal O_F,h)+c_1(0,h)c(\mathcal O_F,f),
\label{eq:product-coefficient-one}\\
c((2),f\cdot h)
=c_1(0,f)c((2),h)+c_1(0,h)c((2),f)
+c(\mathcal O_F,f)c(\mathcal O_F,h).
\label{eq:product-coefficient-two}
\end{gather}
\end{lemma}

\begin{proof}
The proof is very similar to that of Lemma 4.3 in \cite{Joshi-Zhang2019Hilbert}. Take $\nu=1$ in \eqref{eq:component-product-convolution}.
For $\mu=1$, the convolution sum is empty, giving
\eqref{eq:product-coefficient-one}.
For $\mu=2$, the only term in the convolution sum comes from
$2=1+1$ and equals $c(\mathcal O_F,f)c(\mathcal O_F,h)$.
This gives \eqref{eq:product-coefficient-two}.
\end{proof}

We now apply Lemma~\ref{lem:low-product-coefficients} to products
of two normalized full-level Eisenstein eigenforms.
Suppose that $f_i=E_{k_i}(\phi_i,\psi_i)$, with $k_i\geq2$,
and that $g=f_1\cdot f_2$ is a Hecke eigenform.
Write $\chi_i=\phi_i^{-1}\psi_i$, and set
\[
A=L(1-k_1,\chi_1),\qquad
B=L(1-k_2,\chi_2),\qquad
C=L(1-k_1-k_2,\chi_1\chi_2).
\]
We follow the constant-term comparison leading to
\cite[\S5, equations~(5.1)--(5.2)]{Joshi-Zhang2019Hilbert}.
The same argument applies in every degree $n\geq2$, with
$2^{-2}$ replaced by $2^{-n}$ in the Eisenstein constant terms. Thus
\[
g=f_1f_2
=2^{-n}(A+B)
E_{k_1+k_2}(\phi_1\phi_2,\psi_1\psi_2).
\]
Let $F_1,F_2,F_3$ denote the coefficients at $(2)$ of
$f_1$, $f_2$, and $E_{k_1+k_2}(\phi_1\phi_2,\psi_1\psi_2)$,
respectively.
Comparing the principal constant terms and the coefficients
at $(2)$ yields
\begin{gather}
AB=(A+B)C,
\label{eq:eisenstein-product-constant-relation}\\
(A+B)F_3=BF_1+AF_2+2^n.
\label{eq:eisenstein-product-two-coefficient}
\end{gather}

\begin{lemma}
\label{lem:noninert-prime-product-relation}
Under these hypotheses, suppose that $(2)$ is not inert.
For $\mathfrak p\mid(2)$, put $q=N\mathfrak p$ and
$u_i=\chi_i(\mathfrak p)$, with $u_3=u_1u_2$. These quantities satisfy
\begin{equation}
\label{eq:noninert-prime-product-relation}
Bu_1q^{k_1-1}+Au_2q^{k_2-1}
=(A+B)u_3q^{k_1+k_2-1}.
\end{equation}
\end{lemma}
\begin{proof}
Here $q<2^n$. Choose $t_\nu$ with
$[t_\nu]=[\mathfrak p]^{-1}$. By
\eqref{eq:indecomposable-norm-bound}, the convolution sum is empty,
and $\phi_i^{-1}(t_\nu)=\phi_i(\mathfrak p)$.
Substitution of \eqref{eq:eisenstein-fourier-and-constant-terms}
into \eqref{eq:component-product-convolution}, followed by
cancellation of the common character factor, gives
\[
B(1+u_1q^{k_1-1})+A(1+u_2q^{k_2-1})
=(A+B)(1+u_3q^{k_1+k_2-1}),
\]
which proves the identity.
\end{proof}

For the finite indecomposability tests, let $b_1,\ldots,b_n$ be
a $\mathbb Z$-basis of a component lattice $t$, and let
$b_1^\vee,\ldots,b_n^\vee$ be its trace-dual basis.
If $x=\sum_jm_jb_j$ and $0\ll x\ll\mu$, then
$m_j=\operatorname{Tr}(xb_j^\vee)$, which gives
\begin{equation}
\label{eq:prelim-trace-dual-bounds}
\sum_{v(b_j^\vee)<0}v(\mu)v(b_j^\vee)
<m_j<\sum_{v(b_j^\vee)>0}v(\mu)v(b_j^\vee).
\end{equation}
Choose rational lower and upper bounds for these endpoints.
They give a finite box containing every possible integer tuple
$(m_1,\ldots,m_n)$. Checking $0\ll x\ll\mu$ for each tuple
determines all positive decompositions.

\subsection{Petersson and Rankin--Selberg normalizations}

We next fix the Petersson and Rankin--Selberg conventions used for
Eisenstein weight two. We use the adelic Petersson pairing summed over all narrow
components, with the measure normalization of
\cite[\S1]{Panchishkin1989}. The pairing is linear in its first
argument. On a cuspidal space with fixed central character, the
Hecke operators admit an orthogonal eigenbasis.

For normalized cuspidal eigenforms $f\in S_k(\omega_f)$ and
$h\in S_\ell(\omega_h)$, let $\pi_f,\pi_h$ be their unitary
automorphic representations and $\widetilde\pi_f$ the contragredient.
The local parameters for $\pi_f$ are obtained by dividing the roots
of the arithmetic Hecke polynomial by $N(\mathfrak p)^{(k-1)/2}$.
For automorphic representations, $L(s,\cdot)$ denotes the finite
part of the $L$-function. For a narrow class character $\xi$,
the local recurrences \eqref{eq:prelim-prime-power-recurrence} give
\begin{equation}
\label{eq:prelim-rankin-dirichlet-series}
\sum_{\mathfrak a}
\frac{\overline{c(\mathfrak a,f)}c(\mathfrak a,h)\xi(\mathfrak a)}
{N(\mathfrak a)^s}
=\frac{L(u,\widetilde\pi_f\times(\pi_h\otimes\xi))}
{L_F(2u,\omega_f^{-1}\omega_h\xi^2)},
\qquad u=s-\frac{k+\ell-2}{2},
\end{equation}
initially for $\operatorname{Re}(u)>1$, and subsequently by
meromorphic continuation
\cite[\S0.1]{JacquetPiatetskiShapiroShalika1983}.
For $k=\ell+2$, the shift is $u=s-\ell$, also when $\ell=1$.
The pairing formula in
Section~\ref{sec:eisenstein-cuspidal-weight-two-products} involves
the value at $u=1$.
We prove its nonvanishing there.

\subsection{Computer-assisted calculations}
\label{subsec:auxiliary-computations}

We use computer-assisted calculations to enumerate fields and
lattice points, evaluate special values, and compute arithmetic
invariants. We provide the source files and recorded outputs in
\cite{EigenformCertificates}. The footnotes identify the files used
for each calculation.

\section{Products of two cuspidal eigenforms}
\label{sec:cuspidal-products}

We begin with products of two cuspidal eigenforms.
If $f$ and $h$ are cuspidal,
equation~\eqref{eq:product-coefficient-one} therefore gives
$c(\mathcal O_F,f\cdot h)=0$.
Since a nonzero full-level Hecke eigenform has a nonzero unit-ideal
coefficient, $f\cdot h$ cannot be a Hecke eigenform in any positive
parallel weight.

\section{Quadratic Eisenstein products}
\label{sec:quadratic-ee-products}

We consider products of two Eisenstein eigenforms of weights at least
two over real quadratic fields. We begin with narrow class number one
and then turn to larger narrow class number.

Throughout this section and Section~\ref{sec:higher-degree-ee-products}, let
$f_i=E_{k_i}(\phi_i,\psi_i)$ be normalized full-level
Eisenstein eigenforms of parallel weights $k_i\geq2$,
for $i=1,2$.
Put
\[
\chi_i=\phi_i^{-1}\psi_i,\quad i=1,2,\qquad
\chi_3=\chi_1\chi_2.
\]

\subsection{Narrow class number one}
\label{sec:quadratic-ee-class-one}

For real quadratic fields with narrow class number one, the
classification in \cite[Theorem~1(1)]{HaoQinZhou2026} gives only the identity
\[
E_4=60E_2^2
\]
over \(\mathbb Q(\sqrt5)\), established in
\cite[Theorem~7.4]{Joshi-Zhang2019Hilbert}.

\subsection{Larger narrow class number}
\label{sec:quadratic-eisenstein-products}

We now turn to real quadratic fields with $h^+(F)>1$. We first record
some facts about quadratic fields of small discriminant that will also
be used in the later quadratic cases.

\begin{lemma}
\label{lem:quadratic-small-fields}
Let \(F\) be a real quadratic field of discriminant \(D_F\).
\begin{enumerate}[label=(\alph*)]
\item If \(h^+(F)>1\), then \(D_F\geq12\).
\item If \(D_F<28\), then \(h(F)=1\). Among these fields,
\(h^+(F)>1\) holds precisely for \(D_F\in\{12,21,24\}\),
and each of these three fields has \(h^+(F)=2\).
\item If \(D_F<43\), \(h^+(F)>1\), and \((2)\) is inert,
then \(D_F=21\).
\end{enumerate}
\end{lemma}

\begin{proof}
The positive fundamental discriminants below \(28\) are
\[
5,\quad8,\quad12,\quad13,\quad17,\quad21,\quad24.
\]
Minkowski's bound \eqref{eq:prelim-arithmetic-bounds}, with
\(\sqrt{D_F}/2<3\), shows that every ideal class contains an integral
ideal of norm one or two. The ideal \((2)\) is inert for
\(D_F=5,13,21\). For the other four discriminants, respectively,
\[
\sqrt2,\qquad 1+\sqrt3,\qquad
\frac{5+\sqrt{17}}2,\qquad 2+\sqrt6
\]
have norms of absolute value two. Together with their conjugates,
they generate all primes above \((2)\), so \(h(F)=1\) in every case.

Recall that \(h^+(F)=h(F)\) if a unit of norm \(-1\) exists,
and \(h^+(F)=2h(F)\) otherwise. Such units for
\(D_F=5,8,13,17\) are, respectively,
\[
\frac{1+\sqrt5}{2},\qquad 1+\sqrt2,\qquad
\frac{3+\sqrt{13}}2,\qquad 4+\sqrt{17}.
\]
For \(D_F=12,21,24\), the equation for a unit of norm \(-1\)
has no solution modulo \(3\). This proves (a) and (b).

Finally, \((2)\) is inert precisely when \(D_F\equiv5\pmod8\).
Beyond the preceding list, the only such fundamental discriminants
below \(43\) are \(29\) and \(37\). Minkowski's bound gives
\(h(F)=1\) for both: there are no ideals of norm two, and for
\(D_F=37\) the primes of norm three are generated by
\((7+\sqrt{37})/2\) and its conjugate. The units
\((5+\sqrt{29})/2\) and \(6+\sqrt{37}\) have norm \(-1\).
Thus both fields have narrow class number one, proving (c).
\end{proof}

The next lemma shows that the critical Hecke \(L\)-values are
positive over a real quadratic field.

\begin{lemma}
\label{lem:quadratic-critical-value-positive}
Let \(F\) be a real quadratic field of discriminant \(D_F\), and
let \(\chi\) be a narrow class character. Suppose that
\(\chi_v(x)=\operatorname{sgn}(x)^r\) at both real places, where
\(r\in\{0,1\}\). If \(k\geq2\) and \(k\equiv r\pmod2\), then we have
\begin{equation*}
L(1-k,\chi)
=\frac{4D_F^{k-1/2}\Gamma(k)^2}{(2\pi)^{2k}}
L(k,\chi^{-1})>0.
\end{equation*}
\end{lemma}
\begin{proof}
Let \(\tau\) be the nontrivial automorphism of \(F/\mathbb Q\).
For every nonzero integral ideal $\mathfrak a$, the ideal
\(\mathfrak a\tau(\mathfrak a)=(N\mathfrak a)\)
has a totally positive generator. Hence the character values satisfy
\[
\chi(\tau(\mathfrak a))
=\chi(\mathfrak a)^{-1}
=\overline{\chi(\mathfrak a)}.
\]
For real \(s>1\), the Euler factor at a rational prime \(p\) is
\[
\begin{cases}
|1-\chi(\mathfrak p)p^{-s}|^{-2},
 & (p)=\mathfrak p\tau(\mathfrak p),\\
(1-p^{-2s})^{-1},
 & (p)\text{ is inert},\\
(1-\chi(\mathfrak p)p^{-s})^{-1},
 & (p)=\mathfrak p^2.
\end{cases}
\]
In the ramified case, \(\chi(\mathfrak p)^2=\chi((p))=1\).
Thus all factors are positive, and absolute convergence of the Euler
product implies
\(L(s,\chi)>0\).

Put
\[
\Lambda(s,\chi)
=D_F^{s/2}\pi^{-(s+r)}
\Gamma\!\left(\frac{s+r}{2}\right)^2L(s,\chi).
\]
Since \(\chi\) has finite conductor one, the functional equation
and local root number formulas in
\cite[\S2.4]{MolinPage2022Hecke} give
\[
\Lambda(1-s,\chi)=W(\chi)\Lambda(s,\chi^{-1}),
\qquad
W(\chi)=i^{2r}\chi(\mathfrak d_F).
\]
As \(\mathfrak d_F=(\sqrt{D_F})\), whose generator has opposite
signs at the two real places, we have
\[
\chi(\mathfrak d_F)=(-1)^r,
\qquad W(\chi)=1.
\]
Evaluating the functional equation at \(s=k\) and applying the
reflection and duplication formulas, using \(k\equiv r\pmod2\),
we obtain
\[
\begin{aligned}
L(1-k,\chi)
&=D_F^{k-\frac12}\pi^{1-2k}
\left(
\frac{\Gamma((k+r)/2)}{\Gamma((1-k+r)/2)}
\right)^2L(k,\chi^{-1})\\
&=\frac{4D_F^{k-\frac12}\Gamma(k)^2}{(2\pi)^{2k}}
L(k,\chi^{-1})>0,
\end{aligned}
\]
where \(L(k,\chi^{-1})>0\) by the Euler product argument above.
\end{proof}
\begin{remark}
For every full-level Eisenstein eigenform \(E_k(\phi,\psi)\),
the character \(\chi=\phi^{-1}\psi\) satisfies
the parity condition in Lemma~\ref{lem:quadratic-critical-value-positive}.
\end{remark}

Since these special values are positive, the constant-term
relation involves no cancellation when both weights are at least two.

\begin{proposition}
\label{prop:quadratic-eisenstein-larger-narrow-class}
Let \(F\) be a real quadratic field with \(h^+(F)>1\).
If \(f_i=E_{k_i}(\phi_i,\psi_i)\), for \(i=1,2\), are
normalized full-level Eisenstein eigenforms of parallel
weights \(k_i\geq2\), then \(f_1\cdot f_2\) is not a Hecke eigenform.
\end{proposition}

\begin{proof}
Suppose that \(f_1\cdot f_2\) is a Hecke eigenform. Put
\[
A=L(1-k_1,\chi_1),\quad
B=L(1-k_2,\chi_2),\quad
C=L(1-k_1-k_2,\chi_3).
\]
All three values are positive by
Lemma~\ref{lem:quadratic-critical-value-positive}.
Thus \eqref{eq:eisenstein-product-constant-relation} gives
\begin{equation}
\label{eq:Eisensteinconstantterm}
0<\frac CA=\frac{B}{A+B}<1.
\end{equation}
On the other hand, \eqref{eq:hecke-L-value-bounds} implies
\[
\frac CA>
\frac{1}{\zeta(k_1+k_2)^2\zeta(k_1)^2}
\prod_{j=k_1}^{k_1+k_2-1}\frac{D_Fj^2}{4\pi^2}.
\]
By Lemma~\ref{lem:quadratic-small-fields}, \(D_F\geq12\),
so every factor in the product exceeds one.
Since \(k_1,k_2\geq2\), retaining the first two factors
and using the monotonicity of \(\zeta(s)\) for \(s>1\), we obtain
\[
\frac CA>
\frac{9D_F^2}{4\pi^4\zeta(4)^2\zeta(2)^2}
=D_F^2\left(\frac{810}{\pi^8}\right)^2
\geq\left(\frac{9720}{\pi^8}\right)^2>1,
\]
since
\(\zeta(2)=\pi^2/6\) and \(\zeta(4)=\pi^4/90\).
This contradicts \eqref{eq:Eisensteinconstantterm}.
\end{proof}

\section{Eisenstein products in higher degree}
\label{sec:higher-degree-ee-products}

We next consider totally real fields of degree $n>2$, allowing arbitrary
narrow class number. We begin with equal weights and then turn to
distinct weights.

\subsection{Equal weights}
\label{sec:higher-degree-eisenstein-products}

We treat equal weights $k\geq3$ and $k=2$ separately.

In the equal-weight case $k_1=k_2=k$, set
\[
A=L(1-k,\chi_1),\qquad
B=L(1-k,\chi_2),\qquad
C=L(1-2k,\chi_3).
\]

\subsubsection{Equal weights at least three}

We begin with $k\geq3$.

For the numerical comparisons below, we use the elementary estimate
\[
\zeta(s)<1+2^{-s}+\int_2^\infty t^{-s}\,dt
=1+2^{-s}+\frac{2^{1-s}}{s-1}\qquad(s>1).
\]
In particular, we have \(\zeta(3)<5/4\), \(\zeta(5)<67/64\), and
\(\zeta(6)<33/32\). We also use
\(\zeta(2)=\pi^2/6\), \(\zeta(4)=\pi^4/90\), and
\(\zeta(6)=\pi^6/945\).

\begin{theorem}
\label{thm:equal-weight-k-at-least-three}
Let \(F\) be a totally real field of degree \(n>2\), and let \(k\geq3\).
If \(f_i=E_k(\phi_i,\psi_i)\), for \(i=1,2\), are normalized
full-level Eisenstein eigenforms, then \(f_1\cdot f_2\) is not a Hecke eigenform.
\end{theorem}

\begin{proof}
Suppose otherwise. First assume that \((2)\) is inert.
Equations~\eqref{eq:eisenstein-product-constant-relation} and
\eqref{eq:eisenstein-product-two-coefficient} give
\[
\frac{C}{AB}=2^{n(2k-2)}(1-2^{-nk}).
\]
Comparing this identity with the lower bound for $C/(AB)$ in
\eqref{eq:hecke-L-value-bounds} gives the necessary inequality
\begin{equation}
\label{eq:equal-weight-inert-necessary-inequality}
\sqrt{D_F}\,H(k)^n<1,
\end{equation}
where $H(k)$ is defined by
\[
H(k)=\frac{\Gamma(2k)}
{2^{2k-1}\Gamma(k)^2\zeta(2k)\zeta(k)^2}.
\]
To contradict this necessary condition, we need a lower bound
for $H(k)$ that holds for all $k\geq3$. For integers $k\geq3$, the ratio satisfies
\[
\frac{H(k+1)}{H(k)}
=\frac{2k+1}{2k}
\frac{\zeta(2k)}{\zeta(2k+2)}
\left(\frac{\zeta(k)}{\zeta(k+1)}\right)^2>1,
\]
so we obtain
\[
H(k)\geq H(3)=\frac{15}{16\zeta(6)\zeta(3)^2}>\frac{32}{55}.
\]

We now show that \eqref{eq:equal-weight-inert-necessary-inequality}
is impossible. For \(n=3\),
the minimum discriminant \(49\) gives
\(\sqrt{D_F}H(k)^3>7(32/55)^3>1\)
\cite[Table~3]{Voight}. For \(n\geq4\), Minkowski's bound
\eqref{eq:prelim-arithmetic-bounds} gives
\[
\sqrt{D_F}H(k)^n>
\frac{n^n}{n!}\left(\frac{32}{55}\right)^n>1.
\]
The expression on the right exceeds one at \(n=4\). The ratio
of successive values is
\((32/55)(1+1/n)^n>64/55\).
This contradicts \eqref{eq:equal-weight-inert-necessary-inequality}.

Next suppose that \((2)\) is not inert. Choose
\(\mathfrak p\mid(2)\) of least residue degree, and put
\(q=N(\mathfrak p)\) and \(u_i=\chi_i(\mathfrak p)\).
Since \(\sum_{\mathfrak q\mid(2)}e_{\mathfrak q}f_{\mathfrak q}=n\)
and \(\sum_{\mathfrak q\mid(2)}e_{\mathfrak q}\geq2\), we have
\[
q\leq2^{\lfloor n/2\rfloor}\leq2^{n-2}.
\]
Equations~\eqref{eq:noninert-prime-product-relation} and
\eqref{eq:eisenstein-product-constant-relation} give
\[
\frac AC(q^ku_1u_2-u_2)=u_1-u_2.
\]
If \(u_1=u_2\), this equation would give \(q^ku_1=1\),
which is impossible because \(q>1\) and \(|u_1|=1\). Thus
\(u_1\ne u_2\).

Let \(K\) be a cyclotomic field containing the character values,
and put \(d=[K:\mathbb Q]\). Written as a quotient, $C/A$ has
denominator \(u_1-u_2\), a nonzero algebraic integer, and numerator
\(q^ku_1u_2-u_2\), whose conjugates have absolute value at most
\(q^k+1\). This gives an upper bound for the norm of $C/A$. To contradict it,
we use \eqref{eq:hecke-L-value-bounds} and
\eqref{eq:hecke-L-galois-equivariance} to bound every conjugate
from below:
\[
\left|\sigma\!\left(\frac CA\right)\right|>R(k),
\qquad
R(k):=D_F^k\left(
\frac{\Gamma(2k)}
{(2\pi)^k\Gamma(k)\zeta(2k)\zeta(k)}
\right)^n
\]
for every embedding \(\sigma:K\hookrightarrow\mathbb C\).
Taking norms yields
\[
R(k)^d<\left|N_{K/\mathbb Q}\!\left(\frac CA\right)\right|
\leq(q^k+1)^d,
\]
and hence we have
\begin{equation}
\label{eq:higher-degree-noninert-bound}
R(k)<q^k+1.
\end{equation}

We first show that this necessary inequality fails at $k=3$.
To compare its two sides in every degree, the elementary zeta
estimates above and Minkowski's bound give
\[
R(3)>\frac{D_F^3}{6^n}\geq
U_n:=6^{-n}\left(\frac{n^n}{n!}\right)^6.
\]
Put \(V_n=2^{3n-6}+1\), so \(q^3+1\leq V_n\).
Since \(U_3=19683/512>9=V_3\), the comparison
\[
\frac{U_{n+1}}{U_n}
=\frac{(1+1/n)^{6n}}6>
\frac{64}{6}>8>\frac{V_{n+1}}{V_n}
\]
shows that \(R(3)>q^3+1\) for every \(n\geq3\).

Finally, for integers \(k\geq3\), the monotonicity of \(\zeta(s)\) gives
\[
\frac{R(k+1)}{R(k)}
>D_F\left(\frac{2k+1}{\pi}\right)^n
>2^n>q>\frac{q^{k+1}+1}{q^k+1}.
\]
Hence \(R(k)/(q^k+1)\) increases with \(k\), contradicting
\eqref{eq:higher-degree-noninert-bound} for every \(k\geq3\).
\end{proof}

\subsubsection{Weight two}

We now turn to weight two.

\begin{theorem}
\label{thm:equal-weight-k-two}
Let \(F\) be a totally real field of degree \(n>2\). A product of two
normalized full-level Eisenstein eigenforms of weight two over
\(F\) is not a Hecke eigenform.
\end{theorem}

\begin{proof}
Suppose otherwise, and use the notation above with \(k=2\).
The characters \(\chi_1,\chi_2,\chi_3\) are totally even, so they
factor through \(\operatorname{Cl}(F)\).

First suppose that \((2)\) is inert.
Equations~\eqref{eq:eisenstein-product-constant-relation} and
\eqref{eq:eisenstein-product-two-coefficient} give
\begin{equation}
\label{eq:weight-two-inert}
A+B=\frac1{4^n-1}=\frac{AB}{C}.
\end{equation}
Since $AB/C=1/(4^n-1)$, the upper bound for $AB/C$ from
\eqref{eq:hecke-L-value-bounds} gives an upper bound on $D_F$.
We will use it to bound the degree and class number.
With $c=\pi^8/9720$, the resulting inequalities are
\begin{equation}
\label{eq:weight-two-root-discriminant}
\begin{aligned}
D_F&<\bigl((4^n-1)c^n\bigr)^2,\\
\delta_F&<16c^2<15.248.
\end{aligned}
\end{equation}
The Odlyzko--Takeuchi bound \eqref{eq:prelim-arithmetic-bounds}
\cite[Proposition~2.3]{TAKEUCHI83} therefore gives \(n\leq12\).

We next show that \(h(F)\leq2\). For \(n\geq5\), apply the Odlyzko--Takeuchi bound in
\eqref{eq:prelim-arithmetic-bounds} to the Hilbert class field \(H\).
Since \(\delta_H=\delta_F\), we obtain \(nh(F)\leq12\).
In degrees three and four,
\eqref{eq:weight-two-root-discriminant} gives
\(D_F<3434.61\) and \(D_F<53621.89\), respectively.
Roblot's tables \cite{RoblotTables} give \(h(F)\leq2\) throughout
these ranges except for the cubic field of discriminant \(2597\).
That field is generated by a root \(a\) of
\(x^3-x^2-9x+8\). Since \(N(a-2)=6\), the ideal \((2)\) cannot
be inert: in that case, the exponent of \(2\) in this norm
would be divisible by three.

For a totally real field \(K\) of degree \(m\), the functional
equation \eqref{eq:prelim-hecke-functional-equation} gives
\((-1)^m\zeta_K(-1)>0\). If \(h(F)=2\), Artin factorization
\cite[p.~4]{CogdellArtin} for its Hilbert class field yields
\[
L(-1,\varepsilon)=\frac{\zeta_H(-1)}{\zeta_F(-1)},
\]
where \(\varepsilon\) is the nontrivial ordinary class character.
Thus \((-1)^nL(-1,\chi)>0\) for every ordinary class character
\(\chi\) under consideration. If \(n\) is odd, then
\(A+B<0\), contrary to \eqref{eq:weight-two-inert}.

For even \(n\), both \(A\) and \(B\) are positive, so
\eqref{eq:hecke-L-value-bounds} gives a lower bound for \(A+B\).
If \(n=4\), the minimum discriminant \(725\)
\cite[Table~3]{Voight} gives
\[
(4^4-1)(A+B)>
2(4^4-1)725^{3/2}\left(\frac3{\pi^4}\right)^4>8.
\]
For even $n\geq6$, we will show that $(4^n-1)(A+B)>1$,
contrary to \eqref{eq:weight-two-inert}. Minkowski's bound gives
\[
(4^n-1)(A+B)>T_n:=
2(4^n-1)\left(\frac{n^n}{n!}\right)^3
\left(\frac3{\pi^4}\right)^n.
\]
Here \(T_6>1.88\). The binomial estimate
\((1+1/n)^n>9/4\) for \(n\geq3\) shows that
\[
\frac{T_{n+1}}{T_n}
=\frac{4^{n+1}-1}{4^n-1}
\left(1+\frac1n\right)^{3n}\frac3{\pi^4}
>\frac{2187}{16\pi^4}>1.
\]
Thus \((4^n-1)(A+B)>1\) for every even \(n\geq4\),
contrary to \eqref{eq:weight-two-inert}.

Finally, suppose that \((2)\) is not inert, and again choose
\(\mathfrak p\mid(2)\) of least residue degree. Put
\(q=N(\mathfrak p)\) and \(u_i=\chi_i(\mathfrak p)\).
Equation~\eqref{eq:noninert-prime-product-relation} and the
constant-term relation \eqref{eq:eisenstein-product-constant-relation}
again force \(u_1\ne u_2\).
The norm comparison giving \eqref{eq:higher-degree-noninert-bound}
is valid for \(k\geq2\). At \(k=2\), it gives
\begin{equation*}
D_F^2\left(\frac{810}{\pi^8}\right)^n
<q^2+1\leq2^n+1<\left(\frac{17}{8}\right)^n.
\end{equation*}
The last inequality follows from the binomial theorem for \(n\geq3\).
Consequently, we obtain
\[
\delta_F^2<\frac{17\pi^8}{6480}<25.
\]
Thus \(\delta_F<5\). In degree three this gives \(D_F<125\),
so \(h(F)=1\) by \cite{RoblotTables}, contrary to \(u_1\ne u_2\).
In degree four it contradicts \(D_F\geq725\).
For \(n\geq5\), \cite[Proposition~2.3]{TAKEUCHI83} gives
\(\delta_F>29.099e^{-8.3185/5}>5.512\), completing the proof.
\end{proof}
\subsection{Distinct weights}
\label{sec:distinct-weight-eisenstein-products}

The following result was stated in
\cite[Theorem~2]{HaoQinZhou2026}. We give a proof that includes the cases
omitted from the original argument.

\begin{theorem}[Hao--Qin--Zhou, {\cite[Theorem~2]{HaoQinZhou2026}}]
\label{thm:distinct-weight-published}
Over a totally real field of degree \(n>2\), a product of two
full-level Eisenstein eigenforms of distinct parallel weights at
least two is not a Hecke eigenform.
\end{theorem}

\begin{proof}[Proof of Theorem~\ref{thm:distinct-weight-published}]
Suppose otherwise, and assume \(k_1>k_2\geq2\). Put
\[
A=L(1-k_1,\chi_1),\qquad B=L(1-k_2,\chi_2),\qquad
C=L(1-k_1-k_2,\chi_3).
\]
By \eqref{eq:eisenstein-product-constant-relation}, we have
\begin{equation}
\label{eq:unequal-weight-contradiction}
1=\left|\frac AB+1\right|\left|\frac CA\right|
\geq\left(\left|\frac AB\right|-1\right)\left|\frac CA\right|.
\end{equation}
We first prove \(\lvert A/B\rvert>11/10\) for all
\(k_1>k_2\geq2\) with \(k_1\geq4\), and then treat
\((k_1,k_2)=(3,2)\).

To contradict \eqref{eq:unequal-weight-contradiction}, we need
lower bounds for both $|A/B|$ and $|C/A|$. In their gamma-factor
ratios, each increase of the weight contributes a factor
$\delta_Fj/(2\pi)$. Writing $r=\delta_F/(2\pi)$, the estimates in
\cite[equation~(2.4)]{You-Zhang2021Hilbert}, recalled in
\eqref{eq:hecke-L-value-bounds}, give
\begin{align*}
\left|\frac AB\right|
&>\left(
\frac{\prod_{j=k_2}^{k_1-1}rj}{\zeta(k_1)\zeta(k_2)}
\right)^n=:R,
\\
\left|\frac CA\right|
&>\left(
\frac{\prod_{j=k_1}^{k_1+k_2-1}rj}
{\zeta(k_1+k_2)\zeta(k_1)}
\right)^n=:S.
\end{align*}
We have used \(\zeta(s)>1\) to discard the zeta factors
in the numerators.
The cubic bound \(D_F\geq49\) \cite[Table~3]{Voight} and the
Odlyzko--Takeuchi bound \eqref{eq:prelim-arithmetic-bounds}
\cite[Proposition~2.3]{TAKEUCHI83} imply \(\delta_F>7/2\)
for every \(n\geq3\).

Now assume $k_1\geq4$. The discriminant bound just proved gives
$r>7/(4\pi)$. We denote the right-hand side by $r_0$.
Every factor $rj$ in the products above exceeds one.
For \(R\), retain the first two factors when \(k_2=2\), or the
first factor when \(k_2\geq3\). Since \(\zeta(s)\) decreases,
the resulting bounds reduce to the pairs \((4,2)\) and \((4,3)\).
Each expression in the minimum below exceeds one. Since
\(n\geq3\), we obtain
\[
R>\min\left\{
\frac{6r_0^2}{\zeta(4)\zeta(2)},
\frac{3r_0}{\zeta(4)\zeta(3)}
\right\}^3>\frac{11}{10}.
\]
Likewise, retaining the first two factors in \(S\) gives
\[
S>\left(\frac{20r_0^2}{\zeta(6)\zeta(4)}\right)^3>100.
\]
These numerical comparisons use the elementary zeta estimates
from Subsection~\ref{sec:higher-degree-eisenstein-products}
and \(\pi<22/7\). Thus
\eqref{eq:unequal-weight-contradiction} would imply
\(1>(R-1)S>10\), a contradiction.

It remains to consider \((k_1,k_2)=(3,2)\).
Lemma~\ref{lem:odd-weight-unit-parity} forces \(n\) to be even.
For \(n\geq6\), the same discriminant bound gives \(\delta_F>7\).
Using \(\zeta(3)<5/4\) and \(\zeta(5)<67/64\), we obtain
\[
\begin{aligned}
R&>\left(\frac{168}{5\pi^3}\right)^n>\frac32,\\
S&>\left(\frac{147}{\pi^2\zeta(5)\zeta(3)}\right)^n>10^n.
\end{aligned}
\]
Again \((R-1)S>1\), contradicting
\eqref{eq:unequal-weight-contradiction}.

Finally, let \(n=4\). Using \(D_F\geq725\)
\cite[Table~3]{Voight}, \(\zeta(3)<5/4\), \(\zeta(5)<67/64\),
and \(\zeta(2)=\pi^2/6\), \eqref{eq:hecke-L-value-bounds} gives
\[
|A|>\frac{16\cdot725^2\sqrt{725}}{625\pi^{12}}
>\frac{16\cdot725^2\cdot26}{625(22/7)^{12}}>\frac13,
\]
\[
\left|\frac CA\right|>725^2\left(\frac{768}{335\pi^2}\right)^4>1000,
\qquad
\left|\frac CB\right|>725^3\left(\frac{1152}{67\pi^5}\right)^4>1000.
\]
For the last two comparisons, replacing \(\pi\) by \(22/7\)
reduces each assertion to a rational inequality. If \((2)\) is not
inert, choose a prime \(\mathfrak p\mid(2)\) of least residue degree,
as in the equal-weight case. Then \(q=N(\mathfrak p)\) belongs to
\(\{2,4\}\). Put \(u_i=\chi_i(\mathfrak p)\).
Equations~\eqref{eq:noninert-prime-product-relation} and
\eqref{eq:eisenstein-product-constant-relation} give
\[
\left|\frac CA\right|
=\left|\frac{q^3u_1u_2-u_2}{qu_1-u_2}\right|
\leq\frac{q^3+1}{q-1}\leq\frac{65}{3},
\]
contrary to \(\lvert C/A\rvert>1000\).

If \((2)\) is inert, \eqref{eq:eisenstein-fourier-and-constant-terms}
gives the three coefficients \(257,17,65537\) at \((2)\). Hence
\eqref{eq:eisenstein-product-two-coefficient} becomes
\[
65520A+65280B=16.
\]
Using \(A+B=AB/C\) from
\eqref{eq:eisenstein-product-constant-relation}, the reverse
triangle inequality gives
\[
\begin{aligned}
16&=|240A+65280(A+B)|\\
&\geq |A|\left(240-65280\left|\frac BC\right|\right)\\
&>\frac13\left(240-\frac{65280}{1000}\right)>16.
\end{aligned}
\]
This is a contradiction.
\end{proof}

\section{Eisenstein products with a weight-one factor}
\label{sec:weight-one-eisenstein-products}

It remains to consider Eisenstein products with a factor of weight one.
We treat the case of exactly one such factor before the case of two.
For \(E=E_1(\phi,\psi)\), put \(\chi=\phi^{-1}\psi\).
The weight-one formula
\cite[Proposition~2.1, equation~(50)]{Dasgupta2011Hilbert} is
\begin{equation}
\label{eq:weight-one-constant-vector}
c_\nu(0,E)=2^{-n}\bigl(
\phi^{-1}(t_\nu)L(0,\chi)
+\psi^{-1}(t_\nu)L(0,\chi^{-1})\bigr).
\end{equation}
Here \(\chi\) is totally odd and unramified at every finite place.
The order-of-vanishing
formula in \cite[p.~441]{Dasgupta2011Hilbert} gives
\(L(0,\chi)\ne0\) and \(L(0,\chi^{-1})\ne0\).
The constant terms of \(E\) are therefore a linear combination
of two distinct characters, with nonzero coefficients.

\begin{proposition}
\label{prop:weight-one-ee}
Let \(F\) be a totally real field of degree at least two. If
\(f_i=E_{k_i}(\phi_i,\psi_i)\), for \(i=1,2\), are normalized
full-level Eisenstein eigenforms with \(k_i\geq1\) and
\(\min(k_1,k_2)=1\), then \(f_1\cdot f_2\) is not a Hecke eigenform.
\end{proposition}

\begin{proof}
The uniqueness argument for the Eisenstein projection following
Lemma~\ref{lem:low-product-coefficients} is valid in every
weight $w\geq2$, since the roots of the Hecke polynomial
have distinct absolute values.
Thus, if a full-level Hecke eigenform $g$ of weight $w$ has a
nonzero constant term, its Eisenstein projection is proportional
to a single $E_w(\phi,\psi)$. Since the cuspidal projection has
zero constant terms, \eqref{eq:eisenstein-fourier-and-constant-terms}
shows that the constant terms of $g$ are a nonzero scalar multiple
of one narrow class character.

If exactly one factor has weight one, its constant terms contain
two distinct characters with nonzero coefficients. The other
factor contributes a single character. Multiplication on each
component preserves the two distinct characters and their nonzero
coefficients, contrary to the preceding claim.

It remains to treat \(k_1=k_2=1\). Put
\(\chi_i=\phi_i^{-1}\psi_i\), \(A_i=L(0,\chi_i)\), and
\(B_i=L(0,\chi_i^{-1})\), and write
\(a=\chi_1^{-1}\), \(b=\chi_2^{-1}\).
Up to the nonzero scalar \(2^{-2n}\) and the common character
factor \((\phi_1\phi_2)^{-1}\), the product constant terms are
\[
\begin{aligned}
W&=(A_1+B_1a)(A_2+B_2b)\\
&=(A_1A_2+B_1B_2ab)+(B_1A_2a+A_1B_2b).
\end{aligned}
\]
All four coefficients are nonzero. The first sum involves totally
even characters, and the second involves totally odd characters.
We check that neither sum vanishes.
If $ab\ne1$, the characters $1$ and $ab$ are distinct,
so their linear independence shows that the even part is nonzero.
If $ab=1$, then $A_2=B_1$ and $B_2=A_1$,
so the even part equals $2A_1B_1\ne0$.
Similarly, the odd part is nonzero when $a\ne b$.
If $a=b$, then $A_2=A_1$ and $B_2=B_1$,
so the odd part equals $2A_1B_1a\ne0$.
Since no character is both totally even and totally odd, $W$ is
a linear combination of at least two distinct characters. Thus
the product constant terms cannot be proportional to a single
character, contrary to the first paragraph of the proof.
\end{proof}

\begin{proof}[Proof of Theorem~\ref{thm:main-ee}]
The classification for weights at least two is established in
Sections~\ref{sec:quadratic-ee-products}--\ref{sec:higher-degree-ee-products}.
Proposition~\ref{prop:weight-one-ee} excludes all remaining
Eisenstein--Eisenstein products.
\end{proof}

\section{Eisenstein--cuspidal products with Eisenstein weight at least three}
\label{sec:eisenstein-cuspidal-weight-at-least-three}

We now turn to Eisenstein--cuspidal products with Eisenstein weight
\(k\geq3\) and cuspidal weight \(\ell\geq1\). We first record two
lemmas used in these cases. For \(k\geq4\), we begin with quadratic
fields, separated according to their narrow class number, and then
consider fields of higher degree. The final subsection treats weight
three.

\subsection{Integrality and local coefficient identities}
\label{subsec:ec-integrality-local-identities}

We first record an integrality lemma and local coefficient identities
for quadratic fields. We state them for \(k\geq2\) and \(\ell\geq1\),
since we will also use them in the weight-two case.

\begin{samepage}
\begin{lemma}
\label{lem:ec-integrality}
Let \(F\) be a totally real field of degree \(n\geq2\), and let
\(f=E_k(\phi,\psi)\) and \(h\in S_\ell(\eta)\) be normalized
full-level Hecke eigenforms with \(k\geq2\) and \(\ell\geq1\).
If \(g=f\cdot h\) is a Hecke eigenform, then
\[
\alpha=c_1(0,f)^{-1}=\frac{2^n}{L(1-k,\phi^{-1}\psi)}
\]
is a nonzero algebraic integer.
\end{lemma}
\end{samepage}

\begin{proof}
The formula for \(\alpha\) follows from
\eqref{eq:eisenstein-fourier-and-constant-terms}. The proof follows the integrality argument in
\cite[\S6]{Joshi-Zhang2019Hilbert}, extended here to arbitrary
degree and cuspidal weight one. 
Since \(h\) is cuspidal, \eqref{eq:product-coefficient-one} shows
that \(G=\alpha g\) is normalized. Equation~\eqref{eq:product-coefficient-two}
then gives
\begin{equation}
\label{eq:quadratic-ec-two-coefficient}
c((2),G)=c((2),h)+\alpha.
\end{equation}

Both coefficients are algebraic integers by
\cite[Proposition~2.2]{Shimura1978Hilbert} in weights at least two
and Lemma~\ref{lem:weight-one-coefficients} when \(\ell=1\).
Their difference is \(\alpha\), proving the claim.
\end{proof}

We next use a prime above \((2)\) to obtain a lower bound for
\(|\alpha|\) when \((2)\) ramifies. The coefficient calculation
in \cite[\S6.1]{Joshi-Zhang2019Hilbert} also works for
cuspidal weight one.

\begin{lemma}
\label{lem:quadratic-ec-ramified-prime}
Under the hypotheses of Lemma~\ref{lem:ec-integrality}, assume
that \(F\) is a real quadratic field. Put \(G=\alpha g\) and
\(\omega=\phi\psi\), and let \(\mathfrak p\mid(2)\) have norm two.
Then the prime coefficient satisfies
\begin{equation}
\label{eq:quadratic-ec-prime-relation}
c(\mathfrak p,G)=\phi(\mathfrak p)c(\mathfrak p,h).
\end{equation}
If \((2)=\mathfrak p^2\), then we have
\begin{equation}
\label{eq:quadratic-ec-ramified-alpha}
\alpha=\eta(\mathfrak p)2^{\ell-1}
\bigl(1-\omega(\mathfrak p)2^k\bigr).
\end{equation}
\end{lemma}

\begin{proof}
Choose \(t_\nu\) with \([t_\nu]=[\mathfrak p]^{-1}\).
Since \(N\mathfrak p=2<4\),
\eqref{eq:indecomposable-norm-bound} shows that the corresponding
exponent is indecomposable. Equation~\eqref{eq:component-product-convolution}
and cuspidality of \(h\) therefore give
\[
c(\mathfrak p,g)=c_\nu(0,f)c(\mathfrak p,h).
\]
By \eqref{eq:eisenstein-fourier-and-constant-terms},
\(c_\nu(0,f)=\phi(\mathfrak p)c_1(0,f)\).
Multiplying by \(\alpha\) proves
\eqref{eq:quadratic-ec-prime-relation}.

Now suppose \((2)=\mathfrak p^2\). Since \(2\) is totally positive,
\eqref{eq:prelim-principal-character} gives \(\phi(\mathfrak p)^2=1\),
so the prime coefficients of the normalized product and the
cuspidal factor have equal squares. Subtract the Hecke relations
\eqref{eq:hecke-prime-square-recurrence} for \(G\) and \(h\),
with central characters \(\omega\eta\) and \(\eta\), respectively.
Equation~\eqref{eq:quadratic-ec-two-coefficient} then gives
\eqref{eq:quadratic-ec-ramified-alpha}, also when \(\ell=1\).
Since character values have absolute value one, the reverse
triangle inequality gives
\(\lvert\alpha\rvert\geq2^{\ell-1}(2^k-1)\geq2^k-1\).
\end{proof}

\subsection{Quadratic fields with narrow class number one}
\label{sec:quadratic-ec-class-one}

We begin with real quadratic fields of narrow class number one and
Eisenstein weight \(k\geq4\).

\begin{theorem}[Hao--Qin--Zhou and Joshi--Zhang]
\label{thm:quadratic-ec-class-one-published}
Let \(F\) be a real quadratic field with \(h^+(F)=1\). If \(E_k\) and
\(h\) are full-level Eisenstein and cuspidal Hecke eigenforms of
parallel weights \(k\geq4\) and \(\ell\geq1\), respectively,
then \(E_k\cdot h\) is not a Hecke eigenform.
\end{theorem}

\begin{proof}
Lemma~\ref{lem:odd-weight-unit-parity} excludes \(\ell=1\),
so assume \(\ell\geq2\).
For \(D_F>5\), the exclusion follows from
\cite[Theorem~1(2)]{HaoQinZhou2026}.
For \(D_F=5\), \cite[Theorem~7.4]{Joshi-Zhang2019Hilbert}
lists only one Eisenstein--cuspidal identity, whose Eisenstein
factor has weight two. Thus no identity occurs when \(k\geq4\).
\end{proof}
\subsection{Quadratic fields with larger narrow class number}
\label{sec:quadratic-eisenstein-cuspidal}

We next consider real quadratic fields with \(h^+(F)>1\) and Eisenstein
weight \(k\geq4\). As in \cite[\S6]{Joshi-Zhang2019Hilbert}, we work
with the reciprocal Eisenstein constant term.

Suppose that \(g=f\cdot h\) is an eigenform as in
Lemma~\ref{lem:ec-integrality}, with \(F\) real quadratic,
\(h^+(F)>1\), and \(k\geq4\). Put
\(\chi=\phi^{-1}\psi\) and \(\alpha=4/L(1-k,\chi)\).
We begin with a bound for the conjugates of \(\alpha\).

\begin{lemma}
\label{lem:quadratic-ec-conjugate-bound}
Every Galois conjugate of $\alpha$ satisfies
\[
\lvert\sigma(\alpha)\rvert
<2\pi\left(\frac{4\pi^2}{D_F}\right)^{k-1/2}
\frac{\zeta(4)^2}{\Gamma(k)^2}.
\]
Denote the right-hand side by $V(D_F,k)$. Equivalently, we can write
\[
V(D_F,k)=\frac{4^k\pi^{2k+8}\sqrt{D_F}}
{8100D_F^k((k-1)!)^2}.
\]
For \(D_F\geq12\), \(V(D_F,k)\) decreases as \(D_F\) increases
and as \(k\) runs through the integers \(k\geq4\). Moreover,
\(V(28,4)<1\),
\(V(21,5)<1\), and \(V(12,4)<14\).
\end{lemma}

\begin{proof}
Galois equivariance
\eqref{eq:hecke-L-galois-equivariance} and the lower bound in
\eqref{eq:hecke-L-value-bounds} give
\[
\lvert\sigma(\alpha)\rvert
\leq2\pi\left(\frac{4\pi^2}{D_F}\right)^{k-1/2}
\frac{\zeta(k)^2}{\Gamma(k)^2\zeta(4k)}
<V(D_F,k).
\]
Here \(\zeta(k)\leq\zeta(4)\), and the second expression for
\(V\) uses \(\zeta(4)=\pi^4/90\).

The expression for \(V\) shows that it decreases with \(D_F\).
Moreover, we have
\[
\frac{V(D_F,k+1)}{V(D_F,k)}
=\frac{4\pi^2}{D_Fk^2}\leq\frac{\pi^2}{48}<1.
\]
Substitution, using \(\pi<22/7\), gives the three stated bounds.
\end{proof}

These bounds leave only \(D_F=12,21,24\). We exclude the first
and third by ramification at \((2)\), and the second by an exact
zeta value.

\begin{theorem}
\label{thm:quadratic-eisenstein-cuspidal}
Let \(F\) be a real quadratic field. If \(f=E_k(\phi,\psi)\) and
\(h\in S_\ell(\eta)\) are normalized full-level Eisenstein and
cuspidal Hecke eigenforms of parallel weights \(k\geq4\) and
\(\ell\geq1\), respectively, then \(f\cdot h\) is not a Hecke eigenform.
\end{theorem}

\begin{proof}
The case \(h^+(F)=1\) is covered by
Subsection~\ref{sec:quadratic-ec-class-one}.
Suppose, therefore, that \(h^+(F)>1\) and that \(f\cdot h\)
is an eigenform.
By Lemma~\ref{lem:ec-integrality}, \(\alpha\) is a nonzero
algebraic integer. If \(D_F\geq28\), Lemma~\ref{lem:quadratic-ec-conjugate-bound}
gives absolute value less than one for every conjugate of
\(\alpha\), contrary to the fact that its norm is a nonzero
integer.
Lemma~\ref{lem:quadratic-small-fields}(b) thus leaves
\(D_F\in\{12,21,24\}\).

For \(D_F=12\) or \(24\), the ideal \((2)\) ramifies.
Lemmas~\ref{lem:quadratic-ec-conjugate-bound}
and~\ref{lem:quadratic-ec-ramified-prime} give
\[
15\leq2^k-1\leq\lvert\alpha\rvert
<V(D_F,k)\leq V(12,4)<14,
\]
which is impossible for every \(k\geq4\).

Finally, let \(D_F=21\). Since \(V(21,5)<1\), the norm argument
excludes all \(k\geq5\), leaving \(k=4\).
Lemma~\ref{lem:quadratic-small-fields}(b) gives \(h(F)=1\).
The character \(\chi=\phi^{-1}\psi\) is totally even and hence
factors through this trivial ordinary class group, so \(\chi=1\).

Let \(\chi_{21}\) be the primitive quadratic Dirichlet character
of conductor \(21\). Formula~\eqref{eq:generalized-bernoulli-special-value}
gives \(B_{4,\chi_{21}}=-1232\), so \(L(-3,\chi_{21})=308\).\footnote{
We compute $B_{4,\chi_{21}}=-1232$ and
$L(-3,\chi_{21})=308$, and independently verify both results;
see \path{e3_cusp_check.sage} and \path{verify_special_values.py}
in \cite{EigenformCertificates}.}
Thus we obtain
\[
\zeta_F(-3)=\zeta(-3)L(-3,\chi_{21})=\frac{77}{30}.
\]
It follows that \(\alpha=120/77\), a nonintegral rational number,
contradicting Lemma~\ref{lem:ec-integrality}.
\end{proof}

\subsection{Higher-degree fields with Eisenstein weight at least four}
\label{sec:higher-degree-eisenstein-cuspidal}

We turn next to fields of degree greater than two, still with
Eisenstein weight \(k\geq4\). We first record computations for two
cubic fields that will be used in the proof.

Let \(F_{49}\) and \(F_{81}\) be the cubic fields defined by
\[
x^3-x^2-2x+1,\qquad x^3-3x-1,
\]
respectively. Their subscripts indicate their discriminants.

\begin{lemma}
\label{lem:higher-degree-ec-exceptional-arithmetic}
The fields \(F_{49}\) and \(F_{81}\) have ordinary class number one,
and their zeta values satisfy
\[
\zeta_{F_{49}}(-3)=\frac{79}{210},\qquad
\zeta_{F_{81}}(-3)=\frac{199}{90}.
\]
\end{lemma}

\begin{proof}
First, the defining polynomials are irreducible and have
discriminants \(49\) and \(81\). A nontrivial index would force
the field discriminant to be at most \(9\), contrary to the
Minkowski bound \(D_F\geq81/4\) in
\eqref{eq:prelim-arithmetic-bounds}. Thus both orders are maximal.
Their reductions modulo \(2\) are irreducible, so neither field
has an ideal of norm two. The ideal-class bound in the same
equation, \(2\sqrt{D_F}/9\leq2\), shows that every ideal class
contains an integral ideal of norm one or two. Thus both fields have ordinary
class number one.

Next, since both cubic fields have square discriminants, they
are cyclic.
The conductor--discriminant formula
\cite[Theorem~3.11]{Washington1997Cyclotomic} gives conductors
\(7\) and \(9\). For \(q\in\{7,9\}\), let \(\chi_q\) be a
primitive cubic Dirichlet character of conductor \(q\). The Dedekind zeta function then factors as
\[
\zeta_{F_{q^2}}(s)
=\zeta(s)L(s,\chi_q)L(s,\overline{\chi}_q).
\]
Formula~\eqref{eq:generalized-bernoulli-special-value} gives
\[
\begin{aligned}
L(-3,\chi_7)L(-3,\overline{\chi}_7)&=\frac{316}{7},\\
L(-3,\chi_9)L(-3,\overline{\chi}_9)&=\frac{796}{3}.
\end{aligned}
\]
Multiplying by \(\zeta(-3)=1/120\) proves the stated values.\footnote{
We compute $\zeta_{F_{49}}(-3)=79/210$ and
$\zeta_{F_{81}}(-3)=199/90$, and independently verify both results;
see \path{e3_cusp_check.sage} and \path{verify_special_values.py}
in \cite{EigenformCertificates}.}
\end{proof}

The following theorem uses these zeta values to exclude the two
remaining cubic fields.

\begin{theorem}
\label{thm:higher-degree-eisenstein-cuspidal}
Let \(F\) be a totally real field of degree \(n>2\), with arbitrary
narrow class number. If \(f=E_k(\phi,\psi)\) and \(h\in S_\ell(\eta)\)
are normalized full-level Eisenstein and cuspidal Hecke eigenforms
of parallel weights \(k\geq4\) and \(\ell\geq1\), respectively,
then \(f\cdot h\) is not a Hecke eigenform.
\end{theorem}

\begin{proof}
Suppose otherwise. Put \(\chi=\phi^{-1}\psi\).
By Lemma~\ref{lem:ec-integrality},
\(\alpha=2^n/L(1-k,\chi)\) is a nonzero algebraic integer.
To bound $\delta_F$, we will show that all conjugates of $\alpha$
have absolute value below one when $\delta_F$ is large.
Their product would then be a nonzero integer of absolute value
below one, which is impossible. Galois equivariance
\eqref{eq:hecke-L-galois-equivariance} and the lower bound in
\eqref{eq:hecke-L-value-bounds} give
\begin{equation}
\label{eq:higher-degree-ec-conjugate-bound}
\lvert\sigma(\alpha)\rvert
\leq\frac{1}{\zeta(n^2k)}
\left(\frac{\sqrt{2\pi}\,\zeta(k)}{\Gamma(k)}
\left(\frac{2\pi}{\delta_F}\right)^{k-1/2}\right)^n.
\end{equation}
Denote the factor inside the $n$-th power by $b_k(\delta_F)$,
and use the same expression to define $b_k(\delta)$ for $\delta>0$.
Since $\zeta(n^2k)>1$, every conjugate satisfies
$|\sigma(\alpha)|<b_k(\delta_F)^n$. If $\delta_F$ were large enough to give $b_k(\delta_F)<1$ for
every $k\geq4$, the product could not be an eigenform.

For every integer \(k\geq4\), the ratio satisfies
\begin{equation}
\label{eq:higher-degree-ec-bound-ratio}
\frac{b_{k+1}(\delta)}{b_k(\delta)}
=\frac{\zeta(k+1)}{\zeta(k)}\frac{2\pi}{\delta k}
<\frac{2\pi}{\delta k}.
\end{equation}
Each \(b_k\) decreases with \(\delta\). If \(\delta\geq51/10\),
this quotient is less than one for every \(k\geq4\).
Using \(\zeta(4)=\pi^4/90\) and \(\pi<22/7\) gives
\(b_4(51/10)<19/20\). Hence \(\delta_F\geq51/10\) would force
every conjugate of \(\alpha\) to have absolute value less than one. This
contradicts its nonzero integral norm. Thus \(\delta_F<51/10\).

For \(n\geq5\), the Odlyzko--Takeuchi bound
\cite[Proposition~2.3]{TAKEUCHI83} gives
\[
\delta_F>29.099e^{-8.3185/n}
\geq29.099e^{-8.3185/5}>\frac{51}{10}.
\]
Degree four is also impossible, since every totally real quartic
field has discriminant at least \(725>(51/10)^4\)
\cite{BuchmannFord1989}.
Consequently, \(n=3\) and \(D_F<(51/10)^3<133\).
The cubic enumeration \cite{Voight} leaves only \(F_{49}\) and
\(F_{81}\).

Since \(n=3\), Lemma~\ref{lem:odd-weight-unit-parity}
forces \(k\) to be even. If \(k\geq6\), then
\(\delta_F\geq49^{1/3}>18/5\), and
\eqref{eq:higher-degree-ec-bound-ratio} gives
\(b_k(\delta_F)\leq b_6(\delta_F)\).
Using \(\zeta(6)=\pi^6/945\) and \(\pi<22/7\) gives
\(b_6(18/5)<1/2\). Equation~\eqref{eq:higher-degree-ec-conjugate-bound}
therefore forces every conjugate of \(\alpha\) to have absolute
value less than one, again contradicting its nonzero integral norm. Only
\(k=4\) remains.

Finally, the totally even character \(\chi\) factors through the
ordinary class group. By
Lemma~\ref{lem:higher-degree-ec-exceptional-arithmetic},
\(h(F)=1\), so \(\chi=1\) and
\(\alpha=8/\zeta_F(-3)\). The same lemma yields
\[
\alpha=
\begin{cases}
1680/79,&F=F_{49},\\
720/199,&F=F_{81}.
\end{cases}
\]
Both values are nonintegral rational numbers, contradicting
Lemma~\ref{lem:ec-integrality}.
\end{proof}
\subsection{Eisenstein weight three}
\label{sec:eisenstein-cuspidal-weight-three}

It remains to treat Eisenstein weight three. We first record
the arithmetic facts needed in the quartic and quadratic cases.

\begin{lemma}
\label{lem:weight-three-quartic-unit-signatures}
Let $F$ be a totally real quartic field with $D_F<3003$.
For $m\in\{15,20,24\}$, let $F_m=\mathbb Q(\zeta_m)^+$,
the maximal real subfield of the $m$-th cyclotomic field.
If $F$ admits a totally odd narrow class character, then
$F\cong F_m$ for one of these values of $m$.
These fields have discriminants $1125,2000,2304$, respectively,
and ordinary class number one. In each case, the unique totally
odd narrow class character is the character $\varepsilon_m$
of $\mathbb Q(\zeta_m)/F_m$.
\end{lemma}

\begin{proof}
The complete enumeration \cite{BuchmannFord1989,Voight}
consists of these three fields and the eight fields listed below.
\footnote{
We check the eleven quartic discriminants; see
\path{e3_cusp_check.sage} in \cite{EigenformCertificates}.}
Each polynomial in the table defines the corresponding field
\cite{LMFDB} and satisfies $P(-1)=-1$. Thus $1+\theta$ is a
unit of norm $-1$ for any root $\theta$ of $P$.
By the character identity on principal ideals
\eqref{eq:prelim-principal-character}, such a unit excludes a totally
odd narrow class character.
\begin{center}
\begingroup
\setlength{\tabcolsep}{6pt}
\renewcommand{\arraystretch}{1.10}
\begin{tabular}{@{}rl@{}}
\toprule
$D_F$ & $P(x)$\\
\midrule
725  & $x^4-x^3-3x^2+x+1$\\
1600 & $x^4-6x^2+4$\\
1957 & $x^4-4x^2-x+1$\\
2048 & $x^4-4x^2+2$\\
2225 & $x^4-x^3-5x^2+2x+4$\\
2525 & $x^4-2x^3-4x^2+5x+5$\\
2624 & $x^4-2x^3-3x^2+2x+1$\\
2777 & $x^4-x^3-4x^2+x+2$\\
\bottomrule
\end{tabular}
\endgroup
\end{center}

Next, we show that $h(F)=1$ for each of the three remaining fields.
Suppose otherwise. The Hilbert class field $H$ is totally real
of degree $4h(F)\geq8$, with
$\delta_H=\delta_F<3003^{1/4}<7.403$, as recalled in
Section~\ref{sec:preliminaries}.
But the Odlyzko--Takeuchi bound
\cite[Proposition~2.3]{TAKEUCHI83} gives
\[
\delta_H>29.099e^{-8.3185/[H:\mathbb Q]}
\geq29.099e^{-8.3185/8}>10.2,
\]
a contradiction.

Finally, the cyclotomic discriminant formula
\cite{Washington1997Cyclotomic} gives
$\lvert D_{\mathbb Q(\zeta_m)}\rvert=D_{F_m}^2$.
Thus $\mathbb Q(\zeta_m)/F_m$ is quadratic, unramified at
every finite place, and ramified at every real place.
Its character $\varepsilon_m$ is therefore totally odd.
Any other totally odd narrow class character differs from
$\varepsilon_m$ by a totally even character, which factors
through the trivial ordinary class group. This proves uniqueness.
\end{proof}

For the quadratic part of the proof, we will need the following
exact special values. The zeta values will also be used in the
weight-two case.

\begin{lemma}
\label{lem:quadratic-exact-special-values}
Let \(F\) be either \(\mathbb Q(\sqrt3)\) or
\(\mathbb Q(\sqrt{21})\), and let \(\varepsilon\) be its unique
nontrivial narrow class character. If \(F=\mathbb Q(\sqrt3)\), then we have
\[
L(-2,\varepsilon)=\frac19,\qquad
\zeta_F(-1)=\frac16.
\]
If \(F=\mathbb Q(\sqrt{21})\), then we have
\[
L(-2,\varepsilon)=\frac{32}{63},\qquad
\zeta_F(-1)=\frac13.
\]
\end{lemma}

\begin{proof}
The narrow Hilbert class fields are
\[
\mathbb Q(i,\sqrt3),\qquad \mathbb Q(\sqrt{-3},\sqrt{-7}),
\]
respectively.
Let \(K\) be the corresponding extension of \(F\).
The conductor--discriminant formula
\cite[Theorem~3.11]{Washington1997Cyclotomic} gives \(D_K=D_F^2\),
so the relative discriminant is one.
Lemma~\ref{lem:quadratic-small-fields}(b) gives narrow class number two.

Artin factorization
\cite[p.~4]{CogdellArtin} gives
\(\zeta_K(s)=\zeta_F(s)L(s,\varepsilon)\).
Comparing the factorizations over \(F\) and \(\mathbb Q\), we find
\[
L(s,\varepsilon)=
\begin{cases}
L(s,\chi_{-4})L(s,\chi_{-3}),&F=\mathbb Q(\sqrt3),\\
L(s,\chi_{-3})L(s,\chi_{-7}),&F=\mathbb Q(\sqrt{21}),
\end{cases}
\]
where \(\chi_e\) denotes the primitive quadratic Dirichlet character
of discriminant \(e\). Applying
\eqref{eq:generalized-bernoulli-special-value} to the Dirichlet
factors gives the two stated values of $L(-2,\varepsilon)$.\footnote{
We compute $L(-2,\varepsilon)$ for
$F=\mathbb Q(\sqrt3)$ and $\mathbb Q(\sqrt{21})$, obtaining
$1/9$ and $32/63$, respectively; see \path{e3_cusp_check.sage}
in \cite{EigenformCertificates}.}
For the zeta values, we use
\(\zeta_F(s)=\zeta(s)L(s,\chi_{D_F})\) and
\(\zeta(-1)=-1/12\), together with the same Bernoulli formula.\footnote{
We compute $\zeta_F(-1)$ for $D_F=12$ and $21$, obtaining
$1/6$ and $1/3$, respectively; see \path{rq_e2_exact.sage}
in \cite{EigenformCertificates}.}
\end{proof}

We now apply these arithmetic facts. For the quartic fields,
evaluating the special values will contradict integrality of the
reciprocal constant term. For quadratic fields, we will also use
the coefficient identities at primes above $(2)$.

\begin{theorem}
\label{thm:eisenstein-cuspidal-weight-three}
Let $F$ be a totally real field of degree $n\geq2$.
If $f=E_3(\phi,\psi)$ and $h\in S_\ell(\eta)$ are normalized
full-level Eisenstein and cuspidal Hecke eigenforms, with
$\ell\geq1$, then $f\cdot h$ is not a Hecke eigenform.
\end{theorem}

\begin{proof}
Suppose otherwise. Put $\chi=\phi^{-1}\psi$ and
$\omega=\phi\psi$. The character $\chi$ is totally odd.
By Lemma~\ref{lem:ec-integrality},
$\alpha=2^n/L(-2,\chi)$ is a nonzero algebraic integer,
and $G=\alpha f\cdot h$ is normalized.
Lemma~\ref{lem:odd-weight-unit-parity} forces $n$ to be even.

First, Galois equivariance
\eqref{eq:hecke-L-galois-equivariance} and
\eqref{eq:hecke-L-value-bounds} give
\begin{equation}
\label{eq:weight-three-conjugate-bound}
\lvert\sigma(\alpha)\rvert
\leq\frac{1}{\zeta(3n^2)}
\left(\frac{4\pi^3\zeta(3)}{\delta_F^{5/2}}\right)^n
<
\left(\frac{4\pi^3\zeta(3)}{\delta_F^{5/2}}\right)^n.
\end{equation}
If $\delta_F\geq(4\pi^3\zeta(3))^{2/5}$, the displayed bound
makes every conjugate of $\alpha$ smaller than one in absolute
value, contradicting its nonzero integral norm. Hence the root discriminant satisfies
\[
\delta_F<\delta_3=(4\pi^3\zeta(3))^{2/5}.
\]
We next bound this threshold numerically to restrict the degree
and obtain explicit discriminant bounds. The integral test gives
\[
\zeta(3)
<
\sum_{j=1}^{25}\frac1{j^3}+\frac1{1250}
<
\frac{12021}{10000}.
\]
Together with $\pi<355/113$, this implies
$\delta_3^2<55$ and $\delta_3^4<3003$.
Since $\delta_3<7.403$, the Odlyzko bounds in
\cite[\S1.1, Table~1]{Voight} exclude $n\geq6$.
Therefore $n=2$ or $4$.

Suppose that $n=4$, so $D_F<3003$.
Lemma~\ref{lem:weight-three-quartic-unit-signatures} gives
$F=F_m$ and $\chi=\varepsilon_m$ for some
$m\in\{15,20,24\}$.
Artin factorization \cite[p.~4]{CogdellArtin} gives
\[
L(s,\varepsilon_m)
=
\frac{\zeta_{\mathbb Q(\zeta_m)}(s)}{\zeta_{F_m}(s)}
=
\prod_\xi
L(s,\xi),
\]
where the product runs over primitive odd Dirichlet characters
whose conductors divide $m$.
We evaluate the primitive Dirichlet factors at $s=-2$
using \eqref{eq:generalized-bernoulli-special-value}.
For $m=15,20$, the two conjugate factors of conductor five
have values $(-4\pm2i)/5$, whose product is $4/5$.
The resulting special values and the corresponding values of
$\alpha=16/L(-2,\varepsilon_m)$ are listed below.\footnote{
We compute $L(-2,\varepsilon_{15})=128/45$,
$L(-2,\varepsilon_{20})=12$ and $L(-2,\varepsilon_{24})=46/3$,
and independently verify the results; see \path{e3_cusp_check.sage}
and \path{verify_special_values.py} in \cite{EigenformCertificates}.}
\begin{center}
\begingroup
\setlength{\tabcolsep}{6.5pt}
\renewcommand{\arraystretch}{1.10}
\begin{tabular}{@{}rrlr@{}}
\toprule
$m$ & $D_{F_m}$ & $L(-2,\varepsilon_m)$ & $\alpha$\\
\midrule
15 & 1125
& $(-2/9)(-16)(4/5)=128/45$
& $45/8$\\
20 & 2000
& $(-1/2)(-30)(4/5)=12$
& $4/3$\\
24 & 2304
& $(-1/2)(-3)(-2/9)(-46)=46/3$
& $24/23$\\
\bottomrule
\end{tabular}
\endgroup
\end{center}
Each value of $\alpha$ is a nonintegral rational number,
contradicting Lemma~\ref{lem:ec-integrality}.

It remains to consider $n=2$, where $D_F<55$.
Since $f$ has odd weight, $h^+(F)>1$.
Lemma~\ref{lem:quadratic-small-fields}(b),(c) excludes the discriminants
$5,8,13,17,29,37$. In the last two cases, $(2)$ is inert.
The units
\[
3+\sqrt{10},\qquad
32+5\sqrt{41},\qquad
\frac{7+\sqrt{53}}2
\]
have norm $-1$ and exclude $D_F=40,41,53$, respectively.
The remaining discriminants are therefore
\[
D_F\in\{12,21,24,28,33,44\}.
\]

For $D_F=33$, write $(2)=\mathfrak p\mathfrak q$.
Since $2$ is totally positive,
\eqref{eq:prelim-principal-character} gives $\phi((2))=1$.
Applying \eqref{eq:quadratic-ec-prime-relation} at both primes
and using Hecke multiplicativity gives
\[
c((2),G)=\phi((2))c((2),h)=c((2),h).
\]
This contradicts \eqref{eq:quadratic-ec-two-coefficient},
since $\alpha\ne0$.

Next, $(2)$ ramifies for $D_F=12,24,28,44$.
Lemma~\ref{lem:quadratic-ec-ramified-prime} gives
$\lvert\alpha\rvert\geq7$.
On the other hand, \eqref{eq:weight-three-conjugate-bound}
and $\zeta(3)<5/4$ give
\[
\lvert\alpha\rvert<\frac{25\pi^6}{D_F^{5/2}}.
\]
For $D_F=28,44$, this upper bound is at most
$25\pi^6/28^{5/2}<6$, a contradiction.

For $D_F=24$, take $\mathfrak p=(2+\sqrt6)$.
Its generator has mixed signs, so total oddness of $\chi$ gives
$\chi(\mathfrak p)=-1$.
Since $\mathfrak p^2=(2)$, we have $\phi(\mathfrak p)^2=1$,
and hence
$\omega(\mathfrak p)=\phi(\mathfrak p)^2\chi(\mathfrak p)=-1$.
Substituting into \eqref{eq:quadratic-ec-ramified-alpha} gives
\[
9\leq\lvert\alpha\rvert
<
\frac{25\pi^6}{24^{5/2}}
<9.
\]
For the upper bounds at $D_F=28,44$ and $D_F=24$, we use
$\pi<22/7$.

The remaining discriminants are $12$ and $21$.
Lemma~\ref{lem:quadratic-small-fields}(b) identifies $\chi$
with the unique nontrivial narrow class character.
The values in Lemma~\ref{lem:quadratic-exact-special-values}
give $\alpha=36$ and $63/8$, respectively.
The value $63/8$ excludes $D_F=21$, since it is not an
algebraic integer, contrary to Lemma~\ref{lem:ec-integrality}.

For $D_F=12$, the generator of $\mathfrak p=(1+\sqrt3)$
has mixed signs, so the same argument gives
$\omega(\mathfrak p)=-1$.
Equation~\eqref{eq:quadratic-ec-ramified-alpha} becomes
\[
36=9\eta(\mathfrak p)2^{\ell-1}.
\]
Taking absolute values forces $\ell=3$, and the equality then
forces $\eta(\mathfrak p)=1$.
But $\eta$ is totally odd in weight three, so the mixed signs of
the generator give $\eta(\mathfrak p)=-1$, a contradiction.
\end{proof}
\section{Eisenstein--cuspidal products with Eisenstein weight two}
\label{sec:eisenstein-cuspidal-weight-two-products}

We now consider full-level products \(E_2(\phi,\psi)\cdot h\)
with a cuspidal factor of parallel weight \(\ell\geq1\).
We begin with a Rankin--Selberg pairing formula over all narrow
components, then treat \(\ell\geq2\) and \(\ell=1\) separately.
Remark~\ref{rmk:c11-comparison-Li} compares our argument with
\cite{Li2026RankinCohen}.

\subsection{Rankin--Selberg pairing over all narrow components}

We begin by constructing the Eisenstein family used in the pairing
formula.

Let $F$ be a totally real field of degree $n\geq2$.
Let $\phi,\psi$ be narrow class characters whose
product is compatible with weight two, and set
$\chi=\phi^{-1}\psi$.
Then $\chi$ is totally even and factors through the ordinary
class group $\operatorname{Cl}(F)$.
Let
\[
h\in S_\ell(\eta),\qquad
f\in S_{\ell+2}(\phi\psi\eta)
\]
be normalized cuspidal Hecke eigenforms, where $\ell\geq1$.
We use a Petersson pairing that is linear in its first argument,
as in the preliminaries. Mok uses the opposite convention
\cite[Definition~7.1]{Mok2009}.

We use the Hecke coefficient normalization of twisting in
\eqref{eq:prelim-character-twist}. The adelic twist of the
Eisenstein family below is scaled to agree with this normalization
at $w=0$. We include this scalar in the constants in the
Rankin identity.
Put $f_\phi=f\otimes\phi^{-1}$. Its central character is
$\chi\eta$, and its coefficients satisfy
\[
\overline{c(\mathfrak a,f_\phi)}
=\phi(\mathfrak a)\overline{c(\mathfrak a,f)}.
\]
The Rankin identity involves the series
\begin{equation}
\label{eq:c11-twisted-series}
D_\phi(s;f,h)
=\sum_{\mathfrak a\subseteq\mathcal O_F}
\overline{c(\mathfrak a,f)}c(\mathfrak a,h)
\phi(\mathfrak a)N(\mathfrak a)^{-s},
\end{equation}
where the sum is over nonzero integral ideals.
At $s=\ell+1$, we use the value given by meromorphic continuation.

We first identify the weight-two Eisenstein series in Mok's
Rankin identity.
Choose a finite idele $D$ representing $\mathfrak d_F$, and use
the full-level compact subgroup in \cite[equation~(2.1)]{Mok2009},
which is conjugate to $K_f$ by $\operatorname{diag}(D,1)$.
We account for this change of normalization in the constants
below.
Let
\[
\mathcal K(w)
=K_2(w,\chi^{-1},\mathcal O_F)
\]
be Mok's kernel in \cite[equation~(3.8)]{Mok2009}.
Average this kernel over the ordinary class group by setting
\[
\mathcal Q(w)(g)
=\frac1{h(F)}\sum_{z\in\operatorname{Cl}(F)}
\chi(z)^{-1}\mathcal K(w)(zg).
\]
Here each $z$ is represented by a finite central idele.
Equation~(3.6) of \cite{Mok2009} shows that the summands do not
depend on this choice.
The family $\mathcal Q(w)$ has central character $\chi$.

\begin{lemma}
\label{lem:c11-endpoint}
The family $\mathcal Q(w)$ is regular at $w=0$, and
$\mathcal Q(0)$ is a nonzero scalar multiple of the
normalized holomorphic Eisenstein eigenform $E_2(1,\chi)$.
In particular, we have
\[
u=c(\mathcal O_F,\mathcal Q(0))\ne0.
\]
The normalized family
\begin{equation}
\label{eq:c11-Hecke-family}
E_2(\phi,\psi;w)
=u^{-1}\mathcal Q(w)\otimes\phi
\end{equation}
therefore satisfies
$E_2(\phi,\psi;0)=E_2(\phi,\psi)$.
\end{lemma}

\begin{proof}
When the finite conductor is one, Mok's primitive kernel does not
depend on the finite character \cite[equations~(3.4)--(3.5)]{Mok2009}.
The matrix relating the two formulas is
\[
\begin{pmatrix}0&D^{-1}\\D&0\end{pmatrix},
\]
which lies in the chosen compact subgroup.
Consequently, equations~(3.8)--(3.9) of \cite{Mok2009} identify
$h(F)\mathcal Q(w)$ with his unnormalized family
$E_2(w;1,\chi)$.
The continuation after equation~(3.5) and the regularity of the
Hecke $L$-factor in equation~(3.8) show that this family is regular
at zero.

Since $F\ne\mathbb Q$, Proposition~3.6 and equation~(3.15)
of \cite{Mok2009} show that $\mathcal Q(0)$ is holomorphic
with inducing characters $(1,\chi)$.
The same formulas show that the unit-ideal coefficient is nonzero.
Dividing by $u$ therefore gives $E_2(1,\chi)$ in our normalization.
Twisting this form by $\phi$ gives the coefficients
\[
\sum_{\mathfrak r\mid\mathfrak m}
\phi(\mathfrak m\mathfrak r^{-1})
\psi(\mathfrak r)N(\mathfrak r),
\]
so \eqref{eq:eisenstein-fourier-and-constant-terms} identifies it
with $E_2(\phi,\psi)$.
\end{proof}

We now use this family to show that
$E_2(\phi,\psi)\cdot h$ has nonzero Petersson pairing with every
Hecke eigenform in $S_{\ell+2}(\phi\psi\eta)$.

\begin{theorem}
\label{thm:c11-all-component-pairing}
Let $F$ be a totally real field of degree $n\geq2$.
Let $\phi,\psi$ be narrow class characters whose
product is compatible with weight two.
Suppose that
$h\in S_\ell(\eta)$ and
$f\in S_{\ell+2}(\phi\psi\eta)$ are normalized cuspidal
Hecke eigenforms of parallel weights, with $\ell\geq1$.
Then we have
\begin{equation}
\label{eq:c11-all-component-pairing}
\langle E_2(\phi,\psi)\cdot h,f\rangle
=\mathcal C
L(1,\widetilde\pi_f\times(\pi_h\otimes\phi))\ne0,
\end{equation}
where $\mathcal C\ne0$ depends only on
the field, weights, characters, and measure normalizations.
\end{theorem}

\begin{proof}
Put
$I(w)=\langle E_2(\phi,\psi;w)\cdot h,f\rangle$.
Twisting the Eisenstein family back by $\phi^{-1}$ replaces
$f$ by $f_\phi$ in the pairing. The scalar from
\eqref{eq:prelim-character-twist} is included in the constant below.

Averaging the kernel over the ordinary
class group does not change the integral.
Indeed, changing variables by $g'=zg$ gives
\[
\langle\mathcal K(w)(z\,\cdot)\cdot h,f_\phi\rangle
=\chi(z)\langle\mathcal K(w)\cdot h,f_\phi\rangle.
\]
The factor $\chi(z)$ cancels the weight $\chi(z)^{-1}$ in the
definition of $\mathcal Q(w)$.

For $\ell\geq2$, the coefficient bound
\eqref{eq:prelim-ramanujan-bound} gives
\[
|c(\mathfrak a,f)c(\mathfrak a,h)|
\leq\tau_F(\mathfrak a)^2N(\mathfrak a)^\ell.
\]
For $\ell=1$, Lemma~\ref{lem:weight-one-coefficients} and
multiplicativity give
$|c(\mathfrak a,h)|\leq\tau_F(\mathfrak a)$, while the
weight-three case of \eqref{eq:prelim-ramanujan-bound} gives
$|c(\mathfrak a,f)|\leq\tau_F(\mathfrak a)N(\mathfrak a)$.
Thus the same product bound holds for every $\ell\geq1$.
For $\operatorname{Re}(w)>0$, the series
\eqref{eq:c11-twisted-series}, with $s=\ell+1+w$,
converges absolutely, because the following identity holds:
\begin{equation*}
\sum_{\mathfrak a}
\frac{\tau_F(\mathfrak a)^2}{N(\mathfrak a)^s}
=\frac{\zeta_F(s)^4}{\zeta_F(2s)}
\qquad (\operatorname{Re}(s)>1).
\end{equation*}
The kernel also converges absolutely there. Since $k=\ell+2>\ell$,
we may apply \cite[Theorem~7.9]{Mok2009} with
$F_1=f_\phi^\rho$ and $F_2=h$. Here $f_\phi^\rho$ denotes the
holomorphic form obtained by complex conjugating the Fourier
coefficients of $f_\phi$, as in \cite[equation~(2.30)]{Mok2009}.
For these forms, the character in Mok's identity is
$\omega_{F_1}\omega_{F_2}=\chi^{-1}$. Thus his kernel is
$\mathcal K(w)$ when $s=\ell+1+w$. With our convention for
the Petersson pairing, his integral is the one defining $I(w)$.
We obtain
\begin{equation}
\label{eq:c11-unfolding}
I(w)=\mathcal A(w)
D_\phi(\ell+1+w;f,h),
\end{equation}
where $\mathcal A(w)$ is independent of
the eigenforms, regular at zero, and nonzero there.
The archimedean factor in \cite[equation~(7.19)]{Mok2009}
is a nonzero fixed multiple of
\[
\frac{\Gamma(\ell+1+w)^n}
{(4\pi)^{n(\ell+1+w)}}.
\]
The remaining Hecke $L$-factors have finite conductor one and
are evaluated at two when $w=0$. Absolute convergence there
shows that they are regular and nonzero.

To pass to $w=0$ in \eqref{eq:c11-unfolding}, we need a bound
on the integrand uniform in $w$.
Near $w=0$, an Eisenstein family at a fixed level has a
meromorphic continuation in a fixed Fr\'echet space of functions
with uniform moderate growth; see
\cite[Theorem~2.3 and the corollary at the end of \S8.6]
{BernsteinLapid2024}.
Averaging over the class group, twisting, and dividing by $u$
preserve this property, so it also holds for
\eqref{eq:c11-Hecke-family}. By Lemma~\ref{lem:c11-endpoint},
the value $E_2(\phi,\psi;w)(g)$ is regular at $w=0$ for
every $g$. Its negative Laurent coefficients therefore vanish
at every $g$, and hence vanish as functions. The family is thus
holomorphic in this Fr\'echet space near zero.
In particular, for a height function $H\geq1$ on a finite union
of Siegel sets, there are $\epsilon>0$, $C>0$, and $N\geq0$
such that
\[
|E_2(\phi,\psi;w)(g)|\leq C H(g)^N
\qquad (|w|\leq\epsilon).
\]
The cuspidal forms $h$ and $f$ are rapidly decreasing on these sets.
Thus the integrand defining $I(w)$ has an integrable majorant
independent of $w$ near zero.
It follows that $I(w)$ is holomorphic near zero, and dominated
convergence gives
\[
I(0)=\langle E_2(\phi,\psi)\cdot h,f\rangle.
\]
We may therefore specialize \eqref{eq:c11-unfolding} at zero.
At $s=\ell+1$, the kernel parameter $s+1-k$ is zero and
$\Gamma(s+1-\ell)=\Gamma(2)=1$.

To check the local identity in
\eqref{eq:prelim-rankin-dirichlet-series}, put
\[
H_r(a,b)=\sum_{j=0}^r a^{r-j}b^j.
\]
These polynomials satisfy
\[
\sum_{r\geq0}H_r(a,b)H_r(c,d)X^r
=\frac{1-abcdX^2}
{(1-acX)(1-adX)(1-bcX)(1-bdX)}.
\]
For distinct roots, summing four geometric series proves the
identity. The repeated-root cases then follow by continuity.\footnote{
We verify the local generating-function identity and
$abcd=\chi(\mathfrak p)^{-1}$ for the twisted local parameters;
see \path{verify_rankin_local_factor.py} in \cite{EigenformCertificates}.}
Apply it to the unitary local parameters of
$\widetilde\pi_f$ and $\pi_h\otimes\phi$.
Their product is
\[
abcd=\omega_f(\mathfrak p)^{-1}\eta(\mathfrak p)
\phi(\mathfrak p)^2=\chi(\mathfrak p)^{-1}.
\]
Thus \eqref{eq:prelim-rankin-dirichlet-series} gives
\[
D_\phi(\ell+1;f,h)
=\frac{L(1,\widetilde\pi_f\times(\pi_h\otimes\phi))}
{L_F(2,\chi^{-1})}.
\]
The denominator is nonzero.
This proves \eqref{eq:c11-all-component-pairing} with the nonzero constant
\[
\mathcal C
=\frac{\mathcal A(0)}{L_F(2,\chi^{-1})}.
\]

Finally, Shahidi's theorem \cite{Shahidi1980} gives nonvanishing
of the Rankin--Selberg value at one in the unitary normalization.
A pole at one would require
$\pi_f\simeq\pi_h\otimes\phi$.
Their holomorphic weights at every real place are $\ell+2$
and $\ell$, respectively, so this is impossible.
This Rankin--Selberg value is therefore finite and nonzero.
\end{proof}

This nonvanishing gives the following dimension criterion.

\begin{theorem}
\label{thm:c11-all-component-dimension}
Under the hypotheses of Theorem~\ref{thm:c11-all-component-pairing}, we have
\[
E_2(\phi,\psi)\cdot h\text{ is a Hecke eigenform}
\quad\Longleftrightarrow\quad
\dim S_{\ell+2}(\phi\psi\eta)=1.
\]
\end{theorem}

\begin{proof}
Lemma~\ref{lem:low-product-coefficients} gives
\[
c(\mathcal O_F,E_2(\phi,\psi)\cdot h)=2^{-n}L_F(-1,\chi)\ne0.
\]
The special value is nonzero by
\eqref{eq:hecke-L-value-bounds}. Thus
$S_{\ell+2}(\phi\psi\eta)$ has positive dimension.
If its dimension is one, Hecke stability makes the product a Hecke eigenform.
Conversely, suppose that the product is a Hecke eigenform and
$\dim S_{\ell+2}(\phi\psi\eta)>1$. An orthogonal Hecke
eigenbasis of this space contains an eigenform orthogonal to the
product.
This contradicts Theorem~\ref{thm:c11-all-component-pairing}.
Hence $\dim S_{\ell+2}(\phi\psi\eta)=1$.
\end{proof}

\subsection{Cuspidal weights at least two}
\label{sec:eisenstein-cuspidal-weight-two}

For \(\ell\geq2\), we first record coefficient identities shared by
the quadratic and higher-degree cases. We then consider narrow class
number one, fields of higher degree, and real quadratic fields with
larger narrow class number.

\subsubsection{Coefficient identities at small ideals}
\label{subsec:c11-small-ideal-identities}

The following identities apply in degree two as well as in higher
degree. Let \(F\) be a totally real field of degree \(n\geq2\), and let
\(E=E_2(\phi,\psi)\) and \(h\in S_\ell(\eta)\) be normalized
full-level Hecke eigenforms, with \(\ell\geq2\). Suppose that
\(E\cdot h\) is a Hecke eigenform. Put
\(\chi=\phi^{-1}\psi\) and \(\omega=\phi\psi\).
By Lemma~\ref{lem:ec-integrality},
\[
\alpha=\frac{2^n}{L_F(-1,\chi)}
\]
is a nonzero algebraic integer. The product \(G=\alpha E\cdot h\)
is normalized.

The norm bound \eqref{eq:indecomposable-norm-bound} rules out
totally positive decompositions for the ideals below.

\begin{proposition}
\label{prop:c11-small-norm}
If \(\mathfrak a\) is a nonzero integral ideal with
\(N\mathfrak a<2^n\), or with \(N\mathfrak a=2^n\) and
\(\mathfrak a\ne(2)\), then \(G=\alpha E\cdot h\) satisfies
\begin{equation}\label{eq:c11-small-norm}
c(\mathfrak a,G)=\phi(\mathfrak a)c(\mathfrak a,h).
\end{equation}
\end{proposition}
\begin{proof}
By \eqref{eq:indecomposable-norm-bound} and its equality case,
the exponent corresponding to \(\mathfrak a\) is indecomposable
in the stated range. Substituting the Eisenstein constant term
from \eqref{eq:eisenstein-fourier-and-constant-terms} into
\eqref{eq:component-product-convolution} gives
\eqref{eq:c11-small-norm}.
\end{proof}

Applying \eqref{eq:c11-small-norm} at a prime ideal and its square
gives the following obstruction.

\begin{corollary}
\label{cor:c11-prime-exclusion}
There is no prime ideal \(\mathfrak p\) with
\(N(\mathfrak p)^2\leq2^n\) and \(\mathfrak p^2\ne(2)\).
\end{corollary}
\begin{proof}
Apply Proposition~\ref{prop:c11-small-norm} at \(\mathfrak p\)
and \(\mathfrak p^2\). The prime-square relation
\eqref{eq:hecke-prime-square-recurrence} gives
\(\chi(\mathfrak p)N(\mathfrak p)^2=1\), which is impossible
since \(|\chi(\mathfrak p)|=1\) and \(N\mathfrak p>1\).
\end{proof}

\subsubsection{Fields with narrow class number one}

We begin with narrow class number one and use the full-level spaces
and Hecke normalization fixed in the preliminaries.

\begin{theorem}
\label{thm:c11-dimension-class-one}
Suppose \(h^+(F)=1\) and \([F:\mathbb Q]>1\). Let
\(h\in S_\ell(1)\) be a normalized cuspidal Hecke eigenform of
parallel weight \(\ell\geq2\). For even \(k\geq2\), the following equivalence holds:
\[
E_k\cdot h\text{ is a Hecke eigenform}
\quad\Longleftrightarrow\quad
\dim S_{\ell+k}(1)=1.
\]
\end{theorem}

\begin{proof}
Since \(h^+(F)=1\), there is a single narrow component and all
narrow class characters are trivial. In particular, \(h\) has trivial
central character and even weight. The full-level group in
\cite{Li2026RankinCohen} uses the convention involving
\(\mathfrak d_F\) described above. Taking \(r=0\) in
\cite[Theorem~3.4 and Corollary~3.7]{Li2026RankinCohen}, with
the Eisenstein normalization in \S3.3 there, shows that \(E_k\cdot h\)
cannot be a Hecke eigenform when \(\dim S_{\ell+k}(1)>1\). For \(k\geq4\),
this obstruction was already proved in
\cite[Proposition~3.3]{ZhangZhou2024}.
By Lemma~\ref{lem:low-product-coefficients} and
\eqref{eq:hecke-L-value-bounds}, the unit-ideal coefficient is
\(2^{-n}\zeta_F(1-k)\ne0\). Thus \(S_{\ell+k}(1)\)
has positive dimension.
If the dimension is one, Hecke stability makes the product a Hecke eigenform.
\end{proof}

\subsubsection{Finite reduction in degree greater than two}

We turn next to fields of degree greater than two. Throughout this
subsubsection, assume \(n>2\) and retain the notation and eigenform
product hypothesis of Subsubsection~\ref{subsec:c11-small-ideal-identities}.
Since \(\chi\) is totally even, it factors through
\(\operatorname{Cl}(F)\). Thus \(h(F)\geq2\) if
\(\chi\ne1\), whereas \(h(F)=1\) forces \(\chi=1\).

For a prime of residue degree one above a rational prime \(p\), the
norm condition is \(p\leq2^{n/2}\). If \((2)\) is not inert,
Corollary~\ref{cor:c11-prime-exclusion} leaves only the possibility
\((2)=\mathfrak p^2\), which the following proposition excludes.

\begin{proposition}
\label{prop:c11-inertness}
If \(E\cdot h\) is a Hecke eigenform, then \((2)\) is inert in \(F\).
\end{proposition}
\begin{proof}
Suppose that \((2)\) is not inert, and choose a prime \(\mathfrak p\) above it
of least residue degree \(f\). Then \(f\leq n/2\), so
Corollary~\ref{cor:c11-prime-exclusion} forces
\((2)=\mathfrak p^2\), \(N\mathfrak p=2^{n/2}\).
Thus \(n=2f\) is even. Put \(q=2^{n/2}\). Combining
\eqref{eq:c11-small-norm} at \(\mathfrak p\),
\eqref{eq:product-coefficient-two} at \(\mathfrak p^2=(2)\),
and \eqref{eq:hecke-prime-square-recurrence}, we obtain
\[
\alpha=\eta(\mathfrak p)q^{\ell-1}
 \bigl(1-\omega(\mathfrak p)q^2\bigr).
\]
Every conjugate therefore has absolute value at least \(q(q^2-1)\),
whereas \eqref{eq:hecke-L-value-bounds}, with \(k=2\), gives
\begin{equation}\label{eq:c11-alpha-bound}
|\sigma(\alpha)|<\left(\frac{2\pi^4}{3}\right)^nD_F^{-3/2}.
\end{equation}

For \(n=6\), the minimum \(D_F\geq300125\)
\cite[Table~3]{Voight} makes the upper bound less than \(457\),
whereas the lower bound is \(504\). For even \(n\geq8\), substituting the discriminant bound in
\cite[Proposition~2.3]{TAKEUCHI83} into
\eqref{eq:c11-alpha-bound} gives
\[
|\sigma(\alpha)|<e^{3\cdot8.3185/2}
\left(\frac{2\pi^4/3}{29.099^{3/2}}\right)^n<226.
\]
The expression in parentheses is less than one, whereas
\(q(q^2-1)\geq4080\).
It remains to consider \(n=4\), for which the two bounds force
\(D_F<4445.506\). We enumerate the totally real quartic fields
under this bound satisfying \((2)=\mathfrak p^2\) and obtain
\[
D_F\in\{1600,2000,2624,3600,4400\}.
\]
Each field has ordinary class number one, so \(\chi=1\).\footnote{
We enumerate the quartic fields with $D_F<4445.506$ and
$(2)=\mathfrak p^2$, finding the five displayed fields, each with
$h(F)=1$; see \path{e2_square_branch_check.sage}
in \cite{EigenformCertificates}.} The
formula for \(\alpha\) above gives \(|\zeta_F(-1)|\leq4/15\),
whereas \eqref{eq:prelim-hecke-functional-equation} gives
\[
\zeta_F(-1)>\frac{D_F^{3/2}}{(2\pi^2)^4}
 \geq\frac{4000}{\pi^8}>\frac4{15},
\]
a contradiction.
\end{proof}

The bound on \(\alpha\) also bounds the root discriminant
and, through the Hilbert class field, the degree and ordinary
class number. Prime factorizations and finite Euler products then
exclude degrees eight and nine, and degree seven when \(h(F)=1\).

\begin{proposition}
\label{prop:c11-finite-boundary}
If \(E\cdot h\) is a Hecke eigenform, then \((2)\) is inert,
\(n\leq9\), and \(nh(F)\leq14\). If \(\chi=1\), then
\(nh(F)\leq9\). The product cannot be a Hecke eigenform for
\(n\geq8\), or for \(n=7\) and \(h(F)=1\).
\end{proposition}
\begin{proof}
Since \(\alpha\) has nonzero integral norm,
\eqref{eq:c11-alpha-bound} gives
\[
\delta_F<\left(\frac{2\pi^4}{3}\right)^{2/3}.
\]
Applying \cite[Proposition~2.3]{TAKEUCHI83} to the totally real Hilbert class
field, of degree \(nh(F)\) and root discriminant \(\delta_F\),
gives \(nh(F)\leq14\). If \(\chi=1\), the functional equation
\eqref{eq:prelim-hecke-functional-equation} and \(\zeta_F(2)>1\)
sharpen the bound to
\(\delta_F<(4\pi^2)^{2/3}\) and hence \(nh(F)\leq9\).
If \(\chi\ne1\), then \(h(F)\geq2\), so \(n\leq7\).
Thus \(n\leq9\) in all cases, and \(h(F)=1\) in degrees eight
and nine.

In degree nine, the bound is \(\delta_F<11.594175\), contradicting
the minimum \(12.869\ldots\) in \cite[Table 3]{Voight}.
In degree eight, the least discriminant is \(282300416\), with root
discriminant \(11.385\ldots\). Here
\(\alpha\in\mathbb Z_{>0}\).
If \(\alpha\geq2\), \eqref{eq:prelim-hecke-functional-equation} gives
\(\delta_F<10.943444\), impossible. Thus \(\alpha=1\) and
\(D_F<326526622.269\). The complete table
\cite[Table 4]{Voight} leaves
\[
282300416,\qquad309593125,\qquad324000000.
\]
The first and third fields have index-one defining polynomials that
reduce modulo two as \((x^4+x+1)^2\) and \((x^4+x^3+1)^2\),
respectively. Thus \((2)=\mathfrak p^2\) in these fields, contrary
to Proposition~\ref{prop:c11-inertness}. In the middle field,
\((2)\) is inert. Evaluating the Euler factors through \(199\) and applying
\eqref{eq:prelim-euler-product-bound}, we obtain
\[
\zeta_F(2)<1.010297(200/199)^8<1.051634,
\]
whereas \(\alpha=1\) requires
\(\zeta_F(2)=(4\pi^2)^8/309593125^{3/2}>1.08315\).\footnote{
We factor $(2)$ in the three degree-eight fields and compute
$P_{199}(F)$ at $D_F=309593125$, obtaining the displayed
zeta bounds; see \path{e2_high_degree_check.sage}
in \cite{EigenformCertificates}.}

For \(n=7\) and \(h(F)=1\), \(\alpha\) is a negative integer.
If \(|\alpha|\geq2\), \eqref{eq:prelim-hecke-functional-equation} gives \(\delta_F<10.853513\), below the
minimum \(20134393^{1/7}=11.05146\ldots\). Consequently
\(\alpha=-1\) and \(D_F<28162990.673\). The complete degree-seven tables
\cite[Theorem 1]{DriverJones2019}, together with defining polynomials
and prime factorizations from \cite{BordeauxTables2008}, leave
\[
20134393,\qquad25164057,\qquad25367689,\qquad28118369.
\]
In the first, second, and fourth fields, the residue degrees are
respectively
\((1,6)\) at 7, \((1,1,5)\) at 3, and \((1,6)\) at 7.
Corollary~\ref{cor:c11-prime-exclusion} applies since
\(7^2,3^2\leq2^7\). For the remaining field, we evaluate the
Euler factors through \(199\) and apply
\eqref{eq:prelim-euler-product-bound} to obtain\footnote{
We factor the indicated primes in the four degree-seven fields
and compute $P_{199}(F)$ at $D_F=25367689$, obtaining the
displayed zeta bounds; see \path{e2_high_degree_check.sage}
in \cite{EigenformCertificates}.}
\[
\zeta_F(2)<1.022705(200/199)^7<1.059227,
\]
whereas \(\alpha=-1\) requires
\(\zeta_F(2)=(4\pi^2)^7/25367689^{3/2}>1.16975\).
\end{proof}

\enlargethispage{3pt}
\subsubsection{Exclusion in degree greater than two}

We now treat the fields remaining after the preceding reduction.
We first record their class-number data and the cuspidal dimension
estimates needed for the final exclusions. We then bound the possible
cuspidal weights using the coefficient at \((4)\).

\begin{lemma}
\label{lem:c11-small-weight-two-spaces}
Let \(F\) be the cubic field of discriminant \(169\) or \(361\),
or the quartic field of discriminant \(725\). Then
\(h(F)=h^+(F)=1\) and \(S_2(1)=0\).
\end{lemma}
\begin{proof}
The field data are as follows.\footnote{
We compute $h(F)=h^+(F)=1$ and find eight unit signatures
for the cubic fields with $D_F=169,361$;
see \path{cubic_weight_two_genus.sage} in \cite{EigenformCertificates}.}
\begin{center}
\begingroup
\setlength{\tabcolsep}{6pt}
\renewcommand{\arraystretch}{1.10}
\begin{tabular}{@{}rlrrr@{}}
\toprule
$D_F$ & Defining polynomial & $h(F)$ & $h^+(F)$ & Unit signatures\\
\midrule
169 & $x^3-x^2-4x-1$ & 1 & 1 & 8\\
361 & $x^3-x^2-6x+7$ & 1 & 1 & 8\\
725 & $x^4-x^3-3x^2+x+1$ & 1 & 1 & 16\\
\bottomrule
\end{tabular}
\endgroup
\end{center}
Each field has $2^n$ unit signatures. Since
$\mathcal O_F^\times/(\mathcal O_F^\times)^2$ also has order $2^n$,
every totally positive unit is a square.\footnote{
We compute $h(F)=h^+(F)=1$ and find sixteen unit signatures
for the quartic field with $D_F=725$;
see \path{quartic_725_weight_two.sage} in \cite{EigenformCertificates}.}

First consider the two cubic fields. Voight's complete table
\cite[Table~4.3]{VoightShimura2009} gives the signatures
\((0;2^3,3)\) and \((0;2,3^3)\), respectively, for the
Shimura curves of quaternion discriminant and level one. Both curves
have genus zero. Since \(h^+(F)=1\) and totally positive units
are squares, Jacquet--Langlands identifies
\(S_2(1)\) with their holomorphic differentials
\cite[Theorem~3.9 and \S5]{DembeleVoight2013}, proving vanishing.

For discriminant $725$, let $B$ be ramified at all four real
places and at no finite place. The weight-two quaternionic module
is the space of functions on the right ideal classes of a maximal
order $\mathcal O$, and its cuspidal part is the complement of
the constants \cite[Definition~3.7]{DembeleVoight2013}.
Thus it suffices to show \(h(\mathcal O)=1\).

The Eichler class-number formula
\cite[Proposition 5.1 and Lemma 5.3]{KirschmerVoight2010} gives
\[
h(\mathcal O)
 =M+\frac12\left(1-\frac12\right)
    +\frac12\left(1-\frac13\right)
    +\frac12\left(1-\frac15\right)
 =M+\frac{59}{60}.
\]
Here $M$ is the mass of $\mathcal O$.

The CM terms correspond to quadratic extensions $K=F(\zeta_{2q})$.
The condition $\varphi(2q)\mid8$, together with
$\varphi(m)^2\geq m/2$, restricts the search to $q\leq64$.
Write $F=\mathbb Q(a)$ with the defining polynomial above and put
$c=a^3-a^2-2a+1$. The degree-two extensions have the following data.
\begin{center}
\begingroup
\setlength{\tabcolsep}{6pt}
\renewcommand{\arraystretch}{1.10}
\begin{tabular}{@{}rlrrr@{}}
\toprule
$q$ & Relative polynomial & $|\mu(K)|$ & $h(K)$ & $N\mathfrak d_{K/F}$\\
\midrule
2 & $t^2+1$ & 4 & 1 & 256\\
3 & $t^2-t+1$ & 6 & 1 & 81\\
5 & $t^2-ct+1$ & 10 & 1 & 25\\
\bottomrule
\end{tabular}
\endgroup
\end{center}
In each row, the norm of the polynomial discriminant equals the
last entry, so $\mathcal O_F[\zeta_{2q}]=\mathcal O_K$.\footnote{
We enumerate the extensions $F(\zeta_{2q})$ for $q\leq64$ and compute
roots of unity, class numbers and relative discriminants for
the quadratic extensions, which occur at $q=2,3,5$;
see \path{quartic_725_weight_two.sage} in \cite{EigenformCertificates}.}
All finite embedding factors are one.
The three terms in the class-number formula are therefore
$1/4$, $1/3$ and $2/5$, whose sum is $59/60$.

For the unit index, take $u\in\mathcal O_K^\times$.
The totally positive unit $u\bar u$ equals $v^2$ for some
$v\in\mathcal O_F^\times$.
Every conjugate of $u/v$ has absolute value one, so by Kronecker's
theorem \cite[Chapter~1]{Washington1997Cyclotomic}, $u/v$ is
a root of unity.
Thus $\mathcal O_K^\times=\mu(K)\mathcal O_F^\times$,
so the unit index is one.

The mass formula \cite[Main Theorem~26.1.5]{Voight2021}
and \eqref{eq:prelim-hecke-functional-equation} at $m=2$ give
\[
0<M=\frac{2\cdot725^{3/2}\zeta_F(2)}{(4\pi^2)^4}
 <\frac{2\cdot725\sqrt{725}\cdot16}{36^4}<1,
\]
where we used \(\zeta_F(2)\leq\zeta(2)^4<16\) and \(\pi>3\).
Thus \(59/60<h(\mathcal O)<119/60\), forcing
\(h(\mathcal O)=1\). The Jacquet--Langlands correspondence
\cite[Theorem~3.9]{DembeleVoight2013} now gives \(S_2(1)=0\).
\end{proof}

For discriminant \(1125\), both narrow classes occur in the
quaternionic module. We use this module to prove vanishing in
weight two and a dimension lower bound in weight five.

\begin{lemma}
\label{lem:c11-discriminant-1125-spaces}
Let \(F=\mathbb Q(a)\), where
\(a^4-a^3-4a^2+4a+1=0\). Then \(D_F=1125\),
\(h(F)=1\), and \(h^+(F)=2\). The nontrivial narrow class character
\(\varepsilon\) is totally odd, and the cuspidal spaces satisfy
\[
S_2(1)=0,\qquad \dim S_5(\varepsilon)\geq5.
\]
\end{lemma}
\begin{proof}
The field has $D_F=1125$, $h(F)=1$ and $h^+(F)=2$,
with the following fundamental units.
\begin{center}
\begingroup
\setlength{\tabcolsep}{8pt}
\renewcommand{\arraystretch}{1.10}
\begin{tabular}{@{}lr@{}}
\toprule
Fundamental unit & Norm\\
\midrule
$-a^3+3a-1$ & 1\\
$a-1$ & 1\\
$a^3+a^2-3a-1$ & 1\\
\bottomrule
\end{tabular}
\endgroup
\end{center}
The signature image has order $2^4h(F)/h^+(F)=8$.\footnote{
We compute $h(F)=1$, $h^+(F)=2$ and fundamental units,
and find eight unit signatures for $D_F=1125$;
see \path{quartic_1125_sufficient.sage} in \cite{EigenformCertificates}.}
Since these units and $-1$ all have norm one, it is precisely
\[
\{(s_1,s_2,s_3,s_4)\in\{\pm1\}^4:s_1s_2s_3s_4=1\}.
\]
Thus \(\varepsilon\) is totally odd and is the unique
character compatible with odd parallel weight.

Take \(B/F\) ramified at all real places and at no finite place.
For any maximal order \(\mathcal O'\), the reduced unit group
\(\Gamma_{\mathcal O'}=\mathcal O'^\times/\mathcal O_F^\times\)
is a finite rotation group. A projective element of order \(m\),
represented by trace \(t\) and norm \(u\), satisfies
\(t^2/u-2=\zeta_m+\zeta_m^{-1}\), so \(\varphi(m)\leq8\).
Moreover, \(\varphi(m)^2\geq m/2\): each prime-power factor
\(\varphi(p^e)^2/p^e=p^{e-2}(p-1)^2\) is at least one,
except the possible factor \(1/2\) at \(p^e=2\).
Thus \(m\leq128\). Among these integers, the condition
\(\varphi(m)\leq8\) gives \(m\leq30\). In the classification of
finite rotation groups, the cyclic and dihedral groups therefore
have order at most $60$, while the tetrahedral, octahedral and
icosahedral groups have orders $12$, $24$ and $60$, respectively.
Hence
\(\#\Gamma_{\mathcal O'}\leq60\)
\cite[Proposition~11.5.2 and Lemma~26.5.1]{Voight2021}.

The local factors at two and three are \((1-2^{-8})^{-1}\)
and \((1-3^{-4})^{-1}\), respectively. Bounding the remaining
factors by those of \(\zeta(2)^4\) gives
\[
\zeta_F(2)<\zeta(2)^4\frac{27}{85}\frac{256}{405}.
\]
Hence the mass formula \cite[Main Theorem~26.1.5]{Voight2021} yields
\[
M=\frac{2\cdot1125^{3/2}\zeta_F(2)}{(2\pi)^8}
<\frac{1125\sqrt{1125}}{826200}<\frac1{20}.
\]
Each mass summand is at least \(1/60\), so there are fewer than
three right ideal classes. The reduced norm maps these classes onto
\(\operatorname{Cl}^+(F)\) \cite[equation~(7.16)]{DembeleVoight2013},
so there are exactly two. The two narrow norm characters span
this weight-two module, so its cuspidal part is zero.
Since $B$ has no finite ramification, the Jacquet--Langlands
correspondence \cite[Theorem~7.28]{DembeleVoight2013} gives
\(S_2(1)=0\).

To obtain the weight-five lower bound, put \(\tau=a^3-3a+1\), with
\(\tau^2-\tau-1=0\).
In \(B=(-1,-1)_F\), choose \(i^2=j^2=-1\), \(ij=-ji\), and set
\[
\begin{aligned}
z&=\frac{\tau+(\tau-1)i+j}{2},\\
\mathcal O&=\mathcal O_F+\mathcal O_Fi+\mathcal O_Fz+\mathcal O_Fiz.
\end{aligned}
\]
This is the scalar extension of the icosian order
\cite[equation~(11.5.8)]{Voight2021}.
Its closure under multiplication follows from
\[
i^2=-1,\qquad z^2=\tau z-1,\qquad
zi=(1-\tau)+\tau i-iz.
\]
For the basis $(1,i,z,iz)$, the reduced-trace pairing
$\operatorname{trd}(xy)$ has matrix
\[
\begin{pmatrix}
2&0&\tau&1-\tau\\
0&-2&1-\tau&-\tau\\
\tau&1-\tau&\tau-1&-1\\
1-\tau&-\tau&-1&-\tau
\end{pmatrix},
\qquad\det=-1.
\]
Thus the order is maximal.
The group generated by $i$ and $z$ consists of $120$ norm-one
elements. Its projective image has order $60$, so the preceding
bound shows that it is the whole reduced unit group
$\Gamma_{\mathcal O}$.\footnote{
We verify the relations and trace matrix of the quaternion order,
and enumerate the $120$ norm-one elements generated by $i$ and $z$;
see \path{quartic_1125_sufficient.sage} and
\path{verify_quaternion_central_character.py}
in \cite{EigenformCertificates}.}

To obtain the dimension bound in weight five, we compute the
invariant subspace for the unit group just determined.
The relevant representation is
\[
W=\bigotimes_{\sigma:F\hookrightarrow\mathbb R}
\operatorname{Sym}^3(\mathbb C^2).
\]
For an element with reduced trace $t$, the character of $W$ is
$N_{F/\mathbb Q}(t^3-2t)$.
Its character values and their multiplicities are
\begin{center}
\begingroup
\setlength{\tabcolsep}{7pt}
\renewcommand{\arraystretch}{1.10}
\begin{tabular}{@{}rrr@{}}
\toprule
Character value & Multiplicity & Contribution\\
\midrule
256 & 2 & 512\\
1 & 88 & 88\\
0 & 30 & 0\\
\midrule
\multicolumn{1}{@{}l}{Total} & 120 & 600\\
\bottomrule
\end{tabular}
\endgroup
\end{center}
Units of $F$ act by $N_{F/\mathbb Q}(u)^3=1$.
Averaging over these $120$ norm-one elements therefore gives\footnote{
We compute the character of $W$ on the $120$ norm-one elements,
obtaining $\dim W^{\Gamma_{\mathcal O}}=5$;
see \path{quartic_1125_sufficient.sage} and
\path{verify_quaternion_central_character.py}
in \cite{EigenformCertificates}.}
\[
\dim W^{\Gamma_{\mathcal O}}=\frac{512+88}{120}=5.
\]
These invariants form the direct summand associated with the
principal ideal class
\cite[equation~(7.27)]{DembeleVoight2013}.

To identify the central character of this summand, we use the algebraic weight action
in \cite[equations~(7.3) and~(7.10)]{DembeleVoight2013}.
A scalar $b\in F^\times$ acts on $W$ by $N_{F/\mathbb Q}(b)^3$,
so a quaternionic form $\Phi$ satisfies
\[
\Phi(bx)=N_{F/\mathbb Q}(b)^{-3}\Phi(x).
\]
Since $h(F)=1$, the ideal associated with a finite central idele
$c$ is principal, so we may write it as $\mathfrak b=(b)$.
Its normalized action, in the inverse translation convention of
\cite[equation~(7.14)]{DembeleVoight2013}, is
\[
(Z_{\mathfrak b}\Phi)(x)
=(N\mathfrak b)^{-3}\Phi(xc^{-1}).
\]
Since the form has full level and $c/b\in\widehat{\mathcal O}_F^\times$,
we obtain
\[
Z_{\mathfrak b}\Phi
=\frac{N_{F/\mathbb Q}(b)^3}{|N_{F/\mathbb Q}(b)|^3}\Phi
=\varepsilon(\mathfrak b)\Phi.
\]
Thus the five-dimensional summand has central character
$\varepsilon$.\footnote{
We verify the scalar and central idele actions, obtaining
$Z_{\mathfrak b}\Phi=\varepsilon(\mathfrak b)\Phi$;
see \path{verify_quaternion_central_character.py}
in \cite{EigenformCertificates}.} In weight five the entire quaternionic module is
cuspidal \cite[\S7]{DembeleVoight2013}, so the Jacquet--Langlands
correspondence \cite[Theorem~7.28]{DembeleVoight2013} gives
\[
\dim S_5(\varepsilon)\geq5.
\]
\end{proof}

For the following six fields, the coefficient at \((4)\) bounds
the cuspidal weight.

\begin{lemma}
\label{lem:c11-dyadic-cutoffs}
Suppose that \(E\cdot h\) is an eigenform and that \(F\) is one of the
fields with discriminants \(49,169,361,725,1125,5125\).
Then the possible cuspidal weights are
\(2,4,6\) at discriminant \(49\), only \(2\) at discriminants
\(169,361,725\), and \(2,3\) at discriminant \(1125\).
The product cannot be a Hecke eigenform at discriminant \(5125\).
\end{lemma}
\begin{proof}
All six fields have ordinary class number one, so $\chi=1$ and
$\alpha=2^n/\zeta_F(-1)$. The prime $(2)$ is inert and has norm $q=2^n$.
To compare the product coefficient at $(4)$ with the prime-square
relation, we need all decompositions $4=x+(4-x)$ into totally
positive algebraic integers and the Eisenstein divisor sums. Write
\[
\begin{aligned}
\mathcal D_F&=\{x\in\mathcal O_F:0\ll x\ll4\},\\
\sigma_1(\mathfrak a)&=\sum_{\mathfrak b\mid\mathfrak a}N\mathfrak b.
\end{aligned}
\]
By \eqref{eq:product-coefficient-two},
\(c((2),G)=c((2),h)+\alpha\). Comparing
\eqref{eq:component-product-convolution} at \((4)\) with
the relation \eqref{eq:hecke-prime-square-recurrence} gives
\[
q^{\ell-1}(q^2-1)
 =\alpha\left(2c((2),h)+\alpha
 -\sum_{x\in\mathcal D_F}c((x),E)c((4-x),h)\right).
\]
The character factors at \((2)\) and \((4)\) are one because their
generators are totally positive. To bound $\ell$, we estimate the
coefficients in this identity using the Eisenstein formula
\eqref{eq:eisenstein-fourier-and-constant-terms} and the Ramanujan
bound \eqref{eq:prelim-ramanujan-bound}. After dividing by
$q^{\ell-1}$, we obtain
\begin{equation}\label{eq:c11-dyadic-bound}
\begin{aligned}
q^2-1\leq A_F\Biggl(&
 \sum_{x\in\mathcal D_F}\sigma_1((x))\tau_F((4-x))
 \left(\frac{\sqrt{N(4-x)}}q\right)^{\ell-1}\\
 &+4q^{-(\ell-1)/2}+A_Fq^{-(\ell-1)}\Biggr),
\end{aligned}
\end{equation}
where \(A_F\geq|\alpha|\). Since \(0\ll4-x\ll4\) implies \(N(4-x)<4^n=q^2\),
every base on the right is less than one. Thus the right-hand side
decreases with \(\ell\).

The special values and reciprocal constants used below are\footnote{
We verify the tabulated values of $\zeta_F(-1)$ and
$2^n/\zeta_F(-1)$ for $D_F=49,169,361,1125$; see \path{verify_special_values.py}
in \cite{EigenformCertificates}.}
\begin{center}
\begingroup
\setlength{\tabcolsep}{7pt}
\renewcommand{\arraystretch}{1.10}
\begin{tabular}{@{}lrrrr@{}}
\toprule
$D_F$ & 49 & 169 & 361 & 1125\\
\midrule
$\zeta_F(-1)$ & $-1/21$ & $-1/3$ & $-1$ & $4/15$\\
$\alpha=2^n/\zeta_F(-1)$ & $-168$ & $-24$ & $-8$ & $60$\\
\bottomrule
\end{tabular}
\endgroup
\end{center}
For the other two fields, we use the finite Euler product
\(P_{199}(F)\). By \eqref{eq:prelim-euler-product-bound},
$P_{199}(F)<\zeta_F(2)$. Since $\pi<22/7$, the functional equation
\eqref{eq:prelim-hecke-functional-equation} gives
\[
|\alpha|=\frac{(4\pi^2)^4}{D_F^{3/2}\zeta_F(2)}
<\frac{(4(22/7)^2)^4}{D_F^{3/2}P_{199}(F)}.
\]
The resulting rational bounds are
\begin{center}
\begingroup
\setlength{\tabcolsep}{7pt}
\renewcommand{\arraystretch}{1.10}
\begin{tabular}{@{}rrrr@{}}
\toprule
$D_F$ & \makecell[c]{$P_{199}(F)$\\Lower bound}
& \makecell[c]{$|\alpha|$\\Upper bound} & $A_F$\\
\midrule
725 & 1.036232 & 120.468280 & 121\\
5125 & 1.102613 & 6.023835 & $61/10$\\
\bottomrule
\end{tabular}
\endgroup
\end{center}
The lower bounds are rounded down and the upper bounds are rounded
up.\footnote{
We compute $P_{199}(F)$ for $D_F=725,5125$ and the
resulting bounds for $|\alpha|$;
see \path{e2_dyadic_bound.sage} and \path{verify_dyadic_bounds.py}
in \cite{EigenformCertificates}.} For the upper bounds, we use the exact rational products together with the bounds
\[
\sqrt{725}>\frac{2692582403}{10^8},\qquad
\sqrt{5125}>\frac{7158910531}{10^8}.
\]

If $0\ll x\ll4$, then $\operatorname{Tr}(x^2)<16n$.
We enumerate the short vectors in the positive trace lattice and
check total positivity to obtain the complete set $\mathcal D_F$.
Substituting these elements into \eqref{eq:c11-dyadic-bound},
we obtain the following strict upper bounds for its right-hand side.\footnote{
We enumerate $\mathcal D_F$ for the six fields and bound the
right-hand side of \eqref{eq:c11-dyadic-bound} at the displayed weights; see \path{e2_dyadic_bound.sage}
and \path{verify_dyadic_bounds.py} in \cite{EigenformCertificates}.}
\begin{center}
\begingroup
\setlength{\tabcolsep}{6.5pt}
\renewcommand{\arraystretch}{1.10}
\begin{tabular}{@{}lrrrrrr@{}}
\toprule
$D_F$ & 49 & 169 & 361 & 725 & 1125 & 5125\\
\midrule
$A_F$ & 168 & 24 & 8 & 121 & 60 & $61/10$\\
$|\mathcal D_F|$ & 9 & 3 & 3 & 7 & 15 & 7\\
$\ell$ & 7 & 4 & 3 & 4 & 4 & 2\\
Upper bound & 34 & 39 & 34 & 138 & 82 & 149\\
$q^2-1$ & 63 & 63 & 63 & 255 & 255 & 255\\
\bottomrule
\end{tabular}
\endgroup
\end{center}
For example, at $D_F=361$ the only ordered decompositions are
$(1,3)$, $(2,2)$, and $(3,1)$. At $\ell=3$, their contribution gives
\[
8\left(\frac{2\cdot27+18\cdot8+28}{64}
+\frac48+\frac8{64}\right)=\frac{133}{4}<63.
\]
The bounds are rounded upwards using rational arithmetic.
Monotonicity excludes each displayed cuspidal weight and all larger
weights. At discriminants
\(49,169,361,725\), the narrow class number is one, so character
compatibility forces \(\ell\) to be even. The listed possibilities
follow.
\end{proof}

The proof reduces to 772 field isomorphism classes. Arithmetic tests leave
21. Indecomposable exponents exclude 15, and weight bounds and cuspidal
dimensions settle the final six. Fields sharing a discriminant are checked sep-
arately.

\begin{theorem}
\label{thm:c11-higher-degree}
Let \(F\) be a totally real field of degree \(n>2\), with arbitrary ordinary
and narrow class numbers. If \(E_2(\phi,\psi)\) is a full-level
Eisenstein Hecke eigenform and \(h\) is a full-level cuspidal
Hecke eigenform of parallel weight at least two, then \(E_2(\phi,\psi)\cdot h\)
is not a Hecke eigenform.
\end{theorem}
\begin{proof}
Suppose that the product is an eigenform, and apply
Proposition~\ref{prop:c11-finite-boundary}.
If the ordinary class number is even, Cauchy's theorem gives a
nontrivial quadratic ordinary class character. Its pullback to the
narrow class group is totally even. If the cuspidal space containing the product were
one-dimensional, twisting by this character as in
\eqref{eq:prelim-character-twist} would give a quadratic self-twist
of its unique cuspidal representation.
By \cite[Proposition~6.5]{LabesseLanglands1979}, this representation
is induced from a Hecke character of the quadratic extension
defined by the twisting character. If a real place split in this
extension, the induced parameter there would be a sum of two
one-dimensional representations, whereas a holomorphic
discrete-series parameter is irreducible. Thus the extension is CM.
The twisting character must therefore be totally odd, a contradiction.
If this cuspidal space has dimension
greater than one, Theorem~\ref{thm:c11-all-component-dimension}
also excludes the product. Thus the ordinary class number is odd,
and Proposition~\ref{prop:c11-finite-boundary} reduces the degree
to \(3\leq n\leq6\).

We enumerate all totally real fields up to the discriminant bounds
below, including the endpoints.
\begin{center}
\begingroup
\setlength{\tabcolsep}{7pt}
\renewcommand{\arraystretch}{1.10}
\begin{tabular}{@{}lrrrr@{}}
\toprule
$n$ & 3 & 4 & 5 & 6\\
\midrule
Bound on $D_F$ & 4218 & 68133 & 209508 & 2429064\\
\bottomrule
\end{tabular}
\endgroup
\end{center}
We obtain these bounds by rounding upwards in
\eqref{eq:c11-alpha-bound}. In degrees five and six, the ordinary
class number is odd, so \(nh(F)\leq14\) forces \(h(F)=1\).
The sharper bound for the trivial character therefore applies. We use the complete totally real field tables of
\cite{BordeauxTables2008}, retaining every isomorphism class
within these bounds. The numbers of fields in degrees three
through six are $143$, $552$, $37$, and $40$, respectively,
for a total of $772$.

The complete quartic table also contains the field generated by
$x^4-19x^2-24x+16$, of discriminant $65808$.
Its ordinary class number is one, so the sharper discriminant bound
excludes it. Its dyadic ramification also contradicts
Proposition~\ref{prop:c11-inertness}, so it does not occur among
the remaining fields.

Applying the class-number and discriminant bounds, inertness of
$(2)$, Corollary~\ref{cor:c11-prime-exclusion}, and integrality of
$\alpha$, we obtain $21$ remaining fields.\footnote{
We apply the stated arithmetic tests to all $772$ fields,
leaving $21$ isomorphism classes;
see \path{e2_candidate_enumeration.sage} and
\path{verify_field_enumeration.py} in \cite{EigenformCertificates}.} When \(h(F)=1\), the bound
\eqref{eq:prelim-euler-product-bound} gives
\(P_{199}(F)<\zeta_F(2)<P_{199}(F)(200/199)^n\). Combined with
\eqref{eq:prelim-hecke-functional-equation}, this gives an interval
for \(|\alpha|\). If this interval contains
no integer, integrality of \(\alpha\) excludes the field.
For the field with \(D_F=6125\), the exact reciprocal constant
is \(60/13\), so integrality excludes this field.

For each of the following fifteen fields, choose a prime
of the indicated norm.
For $\mathfrak a=\mathfrak p^r$, with $r=1,2$, let $m_r$
be the least positive integer in $\mathfrak a$, and put
\(t_r=(m_r)\mathfrak p^{-r}\).
A decomposition of the exponent $m_r$ is equivalent to an element
$x\in t_r$ satisfying $0\ll x\ll m_r$.
The trace-dual bounds \eqref{eq:prelim-trace-dual-bounds} enclose
the coordinates of every such $x$ in a finite integer box $B_r$.
We enumerate all points in each box and check the signs at all
real embeddings. The box sizes are listed below. In each field,
both exponents are indecomposable.
\begin{center}
\begingroup
\small
\setlength{\tabcolsep}{4.5pt}
\renewcommand{\arraystretch}{1.08}
\begin{tabular}{@{}rlrrr@{}}
\toprule
$D_F$ & Defining polynomial & $N\mathfrak p$ & $|B_1|$ & $|B_2|$\\
\midrule
81 & $x^3-3x-1$ & 3 & 18 & 64\\
257 & $x^3-x^2-4x+3$ & 3 & 12 & 48\\
321 & $x^3-x^2-4x+1$ & 3 & 18 & 48\\
697 & $x^3-7x-5$ & 5 & 12 & 64\\
1257 & $x^3-x^2-8x+9$ & 3 & 18 & 18\\
1489 & $x^3-x^2-10x-7$ & 7 & 12 & 80\\
3969 & $x^3-21x-35$ & 3 & 18 & 18\\
2525 & $x^4-2x^3-4x^2+5x+5$ & 5 & 54 & 144\\
4205 & $x^4-x^3-5x^2-x+1$ & 5 & 36 & 256\\
4525 & $x^4-x^3-7x^2+3x+9$ & 5 & 54 & 144\\
8069 & $x^4-x^3-5x^2+5x+1$ & 5 & 54 & 144\\
16317 & $x^4-2x^3-4x^2+5x+1$ & 5 & 54 & 72\\
14641 & $x^5-x^4-4x^3+3x^2+3x-1$ & 11 & 1024 & 1024\\
38569 & $x^5-5x^3+4x-1$ & 7 & 324 & 576\\
371293 & \makecell[l]{$x^6-x^5-5x^4+4x^3$\\$+6x^2-3x-1$}
& 13 & 1296 & 3072\\
\bottomrule
\end{tabular}
\endgroup
\end{center}
The thirty boxes contain $8776$ coordinate vectors in total.
In each field, neither exponent has a totally positive
decomposition in its component lattice.\footnote{
We verify indecomposability of the exponents for
$\mathfrak p$ and $\mathfrak p^2$ in the fourteen fields with
$h(F)=1$; see \path{e2_indecomposable_search.sage}
and \path{verify_lattice_witnesses.py} in \cite{EigenformCertificates}.}
Since these exponents are indecomposable, the proof of
Proposition~\ref{prop:c11-small-norm} gives the same coefficient
identities at $\mathfrak p$ and $\mathfrak p^2$. The relation
\eqref{eq:hecke-prime-square-recurrence} then forces
\[
\chi(\mathfrak p)N(\mathfrak p)^2=1,
\]
which is impossible since $|\chi(\mathfrak p)|=1$ and
$N\mathfrak p>1$.

Fourteen of these fields have ordinary class number one. The fifteenth is
\[
\begin{gathered}
F=\mathbb Q(a),\quad a^3-21a-35=0,\\
D_F=3969,\quad\operatorname{Cl}(F)\simeq\operatorname{Cl}^+(F)\simeq C_3.
\end{gathered}
\]
Here \(\mathfrak p=(3,a+1)\) and \((3)=\mathfrak p^3\).
The component lattices for \(\mathfrak p\) and \(\mathfrak p^2\)
are \(\mathfrak p^2\) and \(\mathfrak p\), respectively, with
exponent \(3\) in each case. Each bounding box contains \(18\)
coordinate vectors, none of which gives a totally positive decomposition.\footnote{
We check both boxes of $18$ vectors for $D_F=3969$, finding
no totally positive decomposition of exponent $3$ in either component lattice;
see \path{e2_allclass_balanced.sage} and
\path{verify_lattice_witnesses.py} in \cite{EigenformCertificates}.}

The six remaining discriminants are
\[
49,\qquad169,\qquad361,\qquad725,\qquad1125,\qquad5125.
\]
All have ordinary class number one.
Lemma~\ref{lem:c11-dyadic-cutoffs} excludes \(5125\) and leaves
only cuspidal weight two for \(169,361,725\), which is ruled out by
Lemma~\ref{lem:c11-small-weight-two-spaces}. It remains to consider
\(49\) and \(1125\).

At discriminant \(49\), the remaining cuspidal weights are \(2,4,6\).
Let \(\mathfrak p\) be the prime above \(7\). In the notation of
\cite{BorisovGunnells2013}, the weight-two cuspidal space for
\(\Gamma(\mathfrak p)\) has basis
\[
2F_0F_1-F_4^2,\qquad 2F_0F_2-F_1^2,\qquad
2F_0F_4-F_2^2.
\]
Let \(g_7\) be the element whose action on this cuspidal space
is described in \cite[Propositions~6.6 and~8.1]{BorisovGunnells2013}.
Its eigenvalues on the displayed basis are
\(\zeta_7^3,\zeta_7^6,\zeta_7^5\).
Since none is one, the full-level invariant space \(S_2\) vanishes.
The same paper gives \(\dim S_8=4\)
\cite[proof of Proposition~6.7]{BorisovGunnells2013}.
Choose nonzero invariant polynomials \(P_4,P_6\) of degrees
four and six. The three weight-six forms \(E_2^3,E_2\cdot P_4,P_6\) are
independent, since there is no
relation below weight eight
\cite[Propositions~6.7 and~7.2]{BorisovGunnells2013}.
The Eisenstein space is one-dimensional, so \(\dim S_6\geq2\).
Theorem~\ref{thm:c11-all-component-dimension} therefore excludes
cuspidal weights four and six.

At discriminant \(1125\), Lemma~\ref{lem:c11-dyadic-cutoffs}
leaves cuspidal weights two and three. Since \(h(F)=1\),
\(\phi=\psi\). The narrow class group has order two, so the
Eisenstein central character is trivial. The compatible cuspidal
characters are therefore \(1\) in weight two and \(\varepsilon\)
in weight three. Lemma~\ref{lem:c11-discriminant-1125-spaces}
excludes a cuspidal factor of weight two. For a factor of weight
three, the same lemma gives \(\dim S_5(\varepsilon)\geq5>1\), so
Theorem~\ref{thm:c11-all-component-dimension} excludes the product.
\end{proof}

\subsubsection{Real quadratic fields with larger narrow class number}

It remains to consider real quadratic fields with \(h^+(F)>1\).
We begin with the following reduction.

\begin{lemma}
\label{lem:c11-quadratic-reduction}
Let \(F\) be a real quadratic field with \(h^+(F)>1\). If
\(E_2(\phi,\psi)\cdot h\) is a Hecke eigenform, where \(h\) is a full-level
cuspidal Hecke eigenform of parallel weight \(\ell\geq2\), then the field satisfies
\begin{equation}\label{eq:c11-quadratic-reduction}
D_F\in\{12,21,24,28,69,77\},\qquad h(F)=1,\qquad h^+(F)=2.
\end{equation}
\end{lemma}

\begin{proof}
Put \(D=D_F\), \(\chi=\phi^{-1}\psi\), \(\omega=\phi\psi\), and
let \(\eta\) be the character of \(h\). The character \(\chi\)
is totally even and therefore factors through the ordinary class group.
By Lemma~\ref{lem:ec-integrality},
\(\alpha=4/L_F(-1,\chi)\) is a nonzero algebraic integer.
By \eqref{eq:product-coefficient-two}, the normalized product
\(G=\alpha E_2(\phi,\psi)\cdot h\) satisfies
\[
c((2),G)=c((2),h)+\alpha.
\]
Proposition~\ref{prop:c11-small-norm} gives the coefficient identity
for every ideal of norm less than four. The same convolution argument
applies whenever the exponent is indecomposable in its component lattice.

For every Galois conjugate \(\rho\) of \(\chi\), the functional
equation \eqref{eq:prelim-hecke-functional-equation} and the Euler
product give
\[
\begin{aligned}
|L_F(-1,\rho)|&=\frac{D^{3/2}}{4\pi^4}|L_F(2,\overline\rho)|,\\
|L_F(2,\rho)|&>
\left(\frac{\zeta(4)}{\zeta(2)}\right)^2=\frac{\pi^4}{225}.
\end{aligned}
\]
Indeed, the factors above each rational prime are bounded below by
\((1+p^{-2})^{-2}\), with strict inequality at a ramified prime.
By \eqref{eq:hecke-L-galois-equivariance},
\(|\sigma(\alpha)|<3600D^{-3/2}\) for every embedding \(\sigma\).

If \((2)=\mathfrak p\mathfrak q\) splits, multiplicativity and the
small-norm coefficient identity give \(c((2),G)=c((2),h)\), a
contradiction. If \((2)=\mathfrak p^2\) ramifies,
\eqref{eq:hecke-prime-square-recurrence} gives
\[
\alpha=\eta(\mathfrak p)2^{\ell-1}
\bigl(1-4\omega(\mathfrak p)\bigr).
\]
Every conjugate has absolute value at least six, so \(D^3<600^2\).
If \((2)\) is inert, its local factor in \(L_F(2,\rho)\) is
\(16/15\), since \(\rho((2))=1\). Replacing the generic lower
factor \(16/25\) improves the bound to
\(|\sigma(\alpha)|<2160D^{-3/2}\). The nonzero integral norm of
\(\alpha\) therefore gives \(D^3<2160^2\). Hence the ramified
and inert cases satisfy, respectively,
\[
D<71.138,\qquad D<167.098.
\]

We now enumerate the fundamental discriminants for which \((2)\)
ramifies and \(D^3<600^2\), or is inert and \(D^3<2160^2\).
Minkowski's bound \(\sqrt D/2\) from
\eqref{eq:prelim-arithmetic-bounds}, together with elements of the
required norms, gives ordinary class number one except at
\(D=40,60,85,165\), where it is two. The unit signatures then
identify the fields with \(h^+(F)>1\) in
Table~\ref{tab:c11-quadratic-candidates}. To compute their constant
terms for \(\chi=1\), we use
\(\zeta_F(-1)=\zeta(-1)L(-1,\chi_D)\) and
\eqref{eq:generalized-bernoulli-special-value}, which give
\[
\zeta_F(-1)=\frac{1}{24D}\sum_{a=1}^{D}\chi_D(a)a^2.
\]
For \(D=12,21\), this recovers the values in
Lemma~\ref{lem:quadratic-exact-special-values}.
For each of the four fields of ordinary class number two, the
nontrivial ordinary class character corresponds to the extension
\(\mathbb Q(\sqrt5,\sqrt{D/5})/F\).
Both \(5\) and \(D/5\) are positive fundamental discriminants.
The biquadratic discriminant is \(D^2\), proving that the extension
is unramified at every place. Artin factorization
\cite[p.~4]{CogdellArtin} gives
\(L_F(s,\chi)=L(s,\chi_5)L(s,\chi_{D/5})\). Formula
\eqref{eq:generalized-bernoulli-special-value} then gives the reciprocal
constant terms in the last column.

\begin{table}[htbp]
\centering
\caption{Candidate real quadratic fields and reciprocal constant terms.}
\label{tab:c11-quadratic-candidates}
\begingroup
\setlength{\tabcolsep}{7.5pt}
\renewcommand{\arraystretch}{1.10}
\begin{tabular}{@{}rcccc@{}}
\toprule
\multicolumn{1}{@{}c}{$D$} & \multicolumn{1}{c}{$h(F)$}
& \multicolumn{1}{c}{$h^+(F)$}
& \multicolumn{1}{c}{$\alpha\ (\chi=1)$}
& \multicolumn{1}{c@{}}{$\alpha\ (\chi\ne1)$}\\
\midrule
12&1&2&24&---\\
21&1&2&12&---\\
24&1&2&8&---\\
28&1&2&6&---\\
40&2&2&\(24/7\)&10\\
44&1&2&\(24/7\)&---\\
56&1&2&\(12/5\)&---\\
60&2&4&2&5\\
69&1&2&2&---\\
77&1&2&2&---\\
85&2&2&\(4/3\)&\(5/2\)\\
93&1&2&\(4/3\)&---\\
133&1&2&\(12/17\)&---\\
141&1&2&\(2/3\)&---\\
165&2&4&\(6/11\)&\(5/6\)\\
\bottomrule
\end{tabular}
\endgroup
\end{table}

Every field in Table~\ref{tab:c11-quadratic-candidates} has ordinary
class number at most two, so \(\chi\) takes values in \(\{\pm1\}\). Hence
\(L_F(-1,\chi)\) and \(\alpha\) are rational. Since \(\alpha\)
is an algebraic integer, it follows that \(\alpha\in\mathbb Z\).
Integrality excludes \(D=44,56,85,93,\allowbreak133,141,165\), and the
ramified lower bound \(|\alpha|\geq6\) excludes \(D=60\).
Apart from the six stated fields, only \(D=40\) remains.\footnote{
We compute the field data and reciprocal constants in
Table~\ref{tab:c11-quadratic-candidates}, obtaining seven
fields after the integrality and ramified-prime tests;
see \path{rq_certificate.py} and \path{rq_e2_exact.sage}
in \cite{EigenformCertificates}.}

For \(D=40\), take \(F=\mathbb Q(\sqrt{10})\),
\(\mathfrak p=(3,\sqrt{10}-1)\) and \(\nu=7+2\sqrt{10}\).
The identity
\[
\nu=-9+2\cdot3(\sqrt{10}-1)+2(\sqrt{10}-1)^2
\]
shows that \(\nu\in\mathfrak p^2\). Since
\(N(\nu)=N(\mathfrak p^2)=9\), we have
\(\mathfrak p^2=(\nu)\), with \(\nu\) totally positive.
If \(x=a+b\sqrt{10}\) and \(\nu-x\) were both totally positive,
their embeddings would imply \(1\leq a\leq6\) and
\(0<a-b\sqrt{10}<7-2\sqrt{10}<1\). Since $a^2>10b^2$ and $a\leq6$, we have $|b|\leq1$.
The second inequality then forces $b=1$ and $a\geq4$, whereas
\(4-\sqrt{10}>7-2\sqrt{10}\). Thus \(\nu\) is indecomposable.
The coefficient formula for ideals of small norm and for
indecomposable exponents therefore gives
\[
c(\mathfrak p,G)=\phi(\mathfrak p)c(\mathfrak p,h),\qquad
c(\mathfrak p^2,G)=\phi(\mathfrak p)^2c(\mathfrak p^2,h).
\]
The relation \eqref{eq:hecke-prime-square-recurrence}, together
with \(\phi(\mathfrak p)^2=1\), then implies \(9\omega(\mathfrak p)=1\), a contradiction.\footnote{
We verify $(3,\sqrt{10}-1)^2=(7+2\sqrt{10})$ and the
total positivity and indecomposability of its generator;
see \path{verify_d21_local.py} in \cite{EigenformCertificates}.}
The six remaining fields have the class numbers asserted in
\eqref{eq:c11-quadratic-reduction}.
\end{proof}

Five fields in Lemma~\ref{lem:c11-quadratic-reduction} are excluded by
the dyadic coefficient identities. For discriminant \(21\),
we also use the coefficients at the primes above \(3\) and \(5\).

\begin{theorem}
\label{thm:c11-quadratic-exclusion}
Let \(F\) be a real quadratic field with \(h^+(F)>1\). If
\(E_2(\phi,\psi)\) and \(h\) are full-level Eisenstein and cuspidal
Hecke eigenforms, respectively, and the parallel weight of \(h\) is
at least two, then \(E_2(\phi,\psi)\cdot h\) is not a Hecke eigenform.
\end{theorem}

\begin{proof}
Suppose that such a product is a Hecke eigenform. Let \(\ell\geq2\) be the
weight of \(h\), and let \(\eta\) be its central character.
Lemma~\ref{lem:c11-quadratic-reduction} gives the discriminants
and class numbers in \eqref{eq:c11-quadratic-reduction}. Thus
$\phi^{-1}\psi=1$ and $\phi^2=1$. Put
\(G=\alpha E_2(\phi,\phi)\cdot h\) and
\(a_m=c((m),h)\). The reciprocal constants computed in the lemma are
\begin{center}
\begingroup
\setlength{\tabcolsep}{7pt}
\renewcommand{\arraystretch}{1.10}
\begin{tabular}{@{}lrrrrrr@{}}
\toprule
$D_F$ & 12 & 21 & 24 & 28 & 69 & 77\\
\midrule
$\alpha=4/\zeta_F(-1)$ & 24 & 12 & 8 & 6 & 2 & 2\\
\bottomrule
\end{tabular}
\endgroup
\end{center}
For \(D_F=12,24,28\), the ramified identity in the proof of
Lemma~\ref{lem:c11-quadratic-reduction} becomes
\[
\alpha=-3\eta(\mathfrak p)2^{\ell-1}.
\]
The constants \(24,6\) force \(\ell=4,2\), respectively. In
either even weight, \(\eta\) is totally even and hence
trivial, making the right side negative. The constant \(8\)
cannot have absolute value \(3\cdot2^{\ell-1}\).

For \(D_F=69\) or \(77\), we have \(\alpha=2\).
Since \(\sqrt{D_F}>4\), the only ordered totally positive decompositions
of \(4\) are \((1,3),(2,2),(3,1)\). The Eisenstein coefficient
\(e_3=c((3),E_2(\phi,\phi))\) is \(13\) or \(10\), respectively,
since \((3)\) ramifies or is inert. The convolution formula
\eqref{eq:component-product-convolution} at \((4)\) and the relation
\eqref{eq:hecke-prime-square-recurrence} give
\[
3a_2+a_3=2-e_3-30\cdot4^{\ell-2}.
\]
The Ramanujan bound \eqref{eq:prelim-ramanujan-bound} gives
\[
|3a_2+a_3|\leq12\cdot2^{\ell-2}+9\cdot3^{\ell-2}
\leq21\cdot3^{\ell-2}.
\]
The identity instead requires
\(|3a_2+a_3|>30\cdot4^{\ell-2}\), a contradiction for every
\(\ell\geq2\).

It remains to consider \(F=\mathbb Q(\sqrt{21})\). Put
\(w=(1+\sqrt{21})/2\), so that \(\mathcal O_F=\mathbb Z[w]\) and
\(w^2-w-5=0\). Here \(\alpha=12\). The prime
\(\mathfrak p_3=(1+w)\) above \((3)\) has norm \(3\), and its
generator has norm \(-3\). Since \(h(F)=1\) and \(h^+(F)=2\), the prime ideal
represents the nontrivial narrow class and
\(\eta(\mathfrak p_3)=(-1)^\ell\). The prime has norm less than four,
so the small-norm coefficient identity gives
\(c(\mathfrak p_3,G)=\phi(\mathfrak p_3)c(\mathfrak p_3,h)\).

The identities \((1+w)^2=3(2+w)\) and
\((2+w)(3-w)=1\) give \(\mathfrak p_3^2=(3)\).
Since \(\phi^2=1\), the relation
\eqref{eq:hecke-prime-square-recurrence} at \(\mathfrak p_3\) implies
\[
c((3),G)-a_3=-8(-1)^\ell3^{\ell-1}.
\]
The only ordered totally positive decompositions of \(3\) in
\(\mathcal O_F\) are \((1,2)\) and \((2,1)\). Indeed, if
\(x=a+bw\) and \(3-x\) are totally positive, then
\(|b|\sqrt{21}<3\), hence \(b=0\). Since
\(c((2),E_2(\phi,\phi))=5\), the convolution formula
\eqref{eq:component-product-convolution} gives
\(c((3),G)-a_3=12(a_2+5)\). Therefore the coefficient at \((2)\) satisfies
\[
a_2=-5-2(-1)^\ell3^{\ell-2}.
\]
The ideal \((2)\) is prime of norm four, so
\eqref{eq:prelim-ramanujan-bound} gives
\(|a_2|\leq2^\ell\). Write \(A_\ell=|a_2|\).
The base cases are
\[
A_2=7>4,\qquad A_4=23>16,\qquad A_5=49>32.
\]
Moreover, the successive values satisfy
\[
A_{\ell+2}-4A_\ell=
\begin{cases}
10\cdot3^{\ell-2}-15>0,&\ell\geq4\text{ even},\\
10\cdot3^{\ell-2}+15>0,&\ell\geq5\text{ odd}.
\end{cases}
\]
Since \(2^{\ell+2}=4\cdot2^\ell\), induction in steps of two
contradicts this bound for every even \(\ell\geq2\) and every
odd \(\ell\geq5\).
Thus \(\ell=3\) and \(a_2=1\).

The same embedding argument shows that the ordered totally positive
decompositions of \(4\) are \((1,3),(2,2),(3,1)\). Comparing
\eqref{eq:component-product-convolution} at \((4)\) with
\eqref{eq:hecke-prime-square-recurrence} at \((2)\) gives
\[
3a_2+a_3=-1-5\cdot4^{\ell-2}.
\]
Consequently \(a_3=-24\), while
\(a_4=a_2^2-4^{\ell-1}=-15\). There are exactly six ordered
totally positive decompositions of $5$, namely
\[
(1,4),(2,3),(3,2),(4,1),(2+w,3-w),(3-w,2+w).
\]
Indeed, write a summand as \(a+bw\). Its two
embeddings lie strictly between zero and five, so
\(|b|\sqrt{21}<5\) and \(b\in\{-1,0,1\}\). These three cases give
the displayed list. The elements \(2+w\) and \(3-w\) are totally
positive units. The Eisenstein coefficients at \((2),(3),(4)\) are
\(5,13,21\), respectively, and the coefficients of both eigenforms
at the unit ideal are one. Formula
\eqref{eq:component-product-convolution} therefore gives\footnote{
We check the totally positive decompositions of $3,4,5$ and the
product coefficient at $(5)$ over $\mathbb Q(\sqrt{21})$, obtaining
$c((5),G)-a_5=-1188$; see \path{verify_d21_local.py}
in \cite{EigenformCertificates}.}
\[
c((5),G)-a_5
=12\bigl(a_4+5a_3+13a_2+21+2\bigr)=-1188.
\]

On the other hand, \((5)=\mathfrak p\mathfrak q\), where
\(\mathfrak p=(w)\) and \(\mathfrak q=(1-w)\) are distinct primes
of norm five. Both generators have norm \(-5\), so both ideals
represent the nontrivial narrow class. Every integral ideal in this
class has norm at least three: the unit ideal is narrow principal,
and there is no ideal of norm two because \((2)\) is inert.

If a positive exponent corresponding to either prime decomposed in its
component lattice, the two summand ideals would lie in the same
narrow class and have norms at least three. Inequality
\eqref{eq:indecomposable-norm-bound} would then give \(\sqrt5\geq\sqrt3+\sqrt3\), a contradiction.
The two exponents are therefore indecomposable, and the product
coefficients satisfy
\[
\begin{aligned}
c(\mathfrak p,G)&=\phi(\mathfrak p)c(\mathfrak p,h),\\
c(\mathfrak q,G)&=\phi(\mathfrak q)c(\mathfrak q,h).
\end{aligned}
\]
Hecke multiplicativity and \(\phi((5))=1\) now imply
\(c((5),G)=a_5\), contradicting the preceding calculation.
\end{proof}

Combining the preceding exclusions with the known classification
for real quadratic fields of narrow class number one gives the
weight-two classification for cuspidal weights at least two.

\begin{corollary}
\label{cor:c11-classification}
Over totally real fields \(F\ne\mathbb Q\), the only full-level
Eisenstein--cuspidal product identity with Eisenstein weight two and
cuspidal parallel weight at least two is
\[
h_8=120E_2\cdot h_6,
\]
over \(\mathbb Q(\sqrt5)\).
\end{corollary}

\begin{proof}
For real quadratic \(F\) with \(h^+(F)=1\), the assertion follows from
\cite[Theorem~3.9]{Li2026RankinCohen}.
Theorem~\ref{thm:c11-quadratic-exclusion} excludes \(h^+(F)>1\),
and Theorem~\ref{thm:c11-higher-degree} excludes \(n>2\).
\end{proof}

\begin{remark}
\label{rmk:c11-comparison-Li}
Li~[24, Theorem~3.4 and Corollaries~3.5--3.7] treats
central-character projections in even weights and arbitrary narrow class
number. For the individual Hecke eigenform \(E_2(\phi,\psi)\),
Theorem~8.2 normalizes the family by its unit-ideal coefficient at \(w=0\).
Uniform moderate growth and cuspidal decay justify specialization.
\end{remark}

\subsection{Cuspidal weight one}

Let \(F\) be totally real of degree \(n\geq2\), put
\(E=E_2(\phi,\psi)\), \(\chi=\phi^{-1}\psi\) and \(\omega=\phi\psi\),
and let \(h\in S_1(\eta)\) be a normalized full-level cuspidal
Hecke eigenform. The character \(\chi\) is totally even.

\begin{lemma}
\label{lem:c12-weight-one-pairing}
Let \(f\in S_3(\omega\eta)\) be a normalized full-level cuspidal
Hecke eigenform. With the Petersson convention of Theorem~8.2,
\begin{equation*}
\langle E\cdot h,f\rangle
=
C L\bigl(1,\widetilde{\pi}_f\times(\pi_h\otimes\phi)\bigr)
\neq0.
\end{equation*}
Here \(C\neq0\) depends only on the field, characters and fixed measure
normalizations. Consequently, \(E\cdot h\) is a Hecke eigenform
if and only if \(\dim S_3(\omega\eta)=1\).
\end{lemma}

\begin{proof}
Apply Theorems~8.2 and~8.3 with \(\ell=1\).
\end{proof}

\begin{lemma}
\label{lem:c12-target-dimension}
Under the preceding hypotheses, \(\dim S_3(\omega\eta)\geq2\).
\end{lemma}

\begin{proof}
Put \(V=S_3(\omega\eta)\) and \(P=E\cdot h\).
By \eqref{eq:eisenstein-fourier-and-constant-terms}, \eqref{eq:product-coefficient-one}, and \eqref{eq:hecke-L-value-bounds},\newline
\(c(\mathcal O_F,P)=2^{-n}L_F(-1,\chi)\neq0\);
in particular, \(V\neq0\).

Suppose first that \(h(F)\) is odd.
The ordinary class character \(\chi\) has odd order \(d\), so
\(\delta=\chi^{(d+1)/2}\) satisfies \(\delta^2=\chi\).
Put \(\xi=\phi\delta\eta^{-1}\).
Twisting preserves full level, holomorphic weight and cuspidality,
and \eqref{eq:prelim-character-twist} gives
\[
H=(h^3)\otimes\xi
\in S_3(\eta^3\xi^2)
=S_3(\phi^2\chi\eta)
=V.
\]
The principal component of \(h\) has unit-ideal coefficient one,
so its cube is nonzero. Twisting is invertible, hence \(H\neq0\).
Repeated application of \eqref{eq:product-coefficient-one} gives \(c(\mathcal O_F,h^3)=0\),
and twisting preserves this zero.
Thus \(H\) and \(P\) are linearly independent.

If \(h(F)\) is even, choose a nontrivial quadratic ordinary class
character \(\varepsilon\). Its pullback is totally even, and twisting
by \(\varepsilon\) preserves \(V\) since \(\varepsilon^2=1\).
The quadratic self-twist argument in the proof of Theorem~8.12 applies:
it uses total evenness and the irreducibility of the real discrete-series
parameter, both valid here in weight three.
It gives \(\dim V\neq1\). Since \(V\neq0\), we have \(\dim V\geq2\).
\end{proof}

Combining the two lemmas excludes these products in every degree.

\begin{theorem}
\label{thm:c12-complete-exclusion}
Let \(F\) be a totally real field of degree \(n\geq2\). If \(E_2(\phi,\psi)\) is a
normalized full-level Eisenstein eigenform and
\(h\in S_1(\eta)\) is a normalized full-level cuspidal
Hecke eigenform, then \(E_2(\phi,\psi)\cdot h\) is not a Hecke eigenform.
\end{theorem}

\begin{proof}
If the product were a Hecke eigenform,
Lemma~\ref{lem:c12-weight-one-pairing} would give
$\dim S_3(\omega\eta)=1$. This contradicts
Lemma~\ref{lem:c12-target-dimension}.
\end{proof}

Together with Section~\ref{sec:eisenstein-cuspidal-weight-at-least-three},
this completes the classification for Eisenstein weight at least two.

\begin{proof}[Proof of Theorem~\ref{thm:main-ec}]
For \(k\geq4\), Theorems~\ref{thm:quadratic-eisenstein-cuspidal}
and~\ref{thm:higher-degree-eisenstein-cuspidal} exclude the products
in degrees two and greater than two, respectively.
Theorem~\ref{thm:eisenstein-cuspidal-weight-three} excludes \(k=3\).
For \(k=2\), Corollary~\ref{cor:c11-classification} covers
\(\ell\geq2\), and Theorem~\ref{thm:c12-complete-exclusion}
excludes \(\ell=1\).
The only identity is therefore
\(h_8=120E_2\cdot h_6\) over \(\mathbb Q(\sqrt5)\).
\end{proof}

\section{A weight-one Eisenstein product identity}
\label{sec:weight-one-products}

We prove the identity in Theorem~3 over \(\mathbb Q(\sqrt3)\),
with an Eisenstein factor of weight one and a cuspidal factor of weight four. We first compute the relevant adelic cuspidal dimensions on the two narrow components.

\begin{lemma}
\label{lem:weight-one-sqrt3-dimensions}
Let \(F=\mathbb Q(\sqrt3)\), and let \(\varepsilon\) be its nontrivial
narrow class character. The full adelic cuspidal spaces satisfy
\begin{equation*}
\dim S_4(1)=2,\qquad \dim S_5(\varepsilon)=1.
\end{equation*}
\end{lemma}

\begin{proof}
Set \(u=2+\sqrt3\). By Lemma~\ref{lem:quadratic-small-fields}, we
have \(h(F)=1\). Every unit has norm one, since the equation
\(a^2-3b^2=-1\) is impossible modulo three.
Thus \(h^+(F)=2\), with narrow classes represented by
\(\mathcal O_F\) and \((\sqrt3)\). In the convention of
Section~\ref{sec:preliminaries}, the two classical groups are
\(\Gamma_0(t_\nu\mathfrak d_F,\mathcal O_F)\).
Since \(\mathfrak d_F=(2\sqrt3)\), the ideals
\(t_\nu\mathfrak d_F\) are represented by \((2\sqrt3)\) and
\((6)\). Diagonal conjugation identifies their modular forms with
forms for \(\mathrm{GL}_2^+(\mathcal O_F)\) on
\(\mathfrak H\times\mathfrak H^-\) and \(\mathfrak H^2\),
respectively, where \(\mathfrak H^-\) is the lower half-plane.
The principal narrow component therefore corresponds to
\(\mathfrak H\times\mathfrak H^-\).

To pass from \(\mathrm{SL}_2(\mathcal O_F)\) to
\(\mathrm{GL}_2^+(\mathcal O_F)\), it suffices to impose
invariance under \(\operatorname{diag}(u,1)\), since this matrix
and \(\mathrm{SL}_2(\mathcal O_F)\) generate the latter group
modulo scalars. Since \(N(u)=1\), its determinant factor
in parallel weight \(k\) is \(N(u)^{k/2}=1\). The required
condition is therefore invariance under \(z\mapsto uz\) on
each domain. Both symmetric and antisymmetric forms are included.
Since every ideal is principal and every unit has positive norm,
the full-level central character compatible with parallel weight
\(k\) is determined by \eqref{eq:prelim-principal-character}:
\[
\eta((a))=\operatorname{sgn}(N(a))^k.
\]
It is \(1\) for even \(k\) and \(\varepsilon\) for odd \(k\).
Adding the dimensions of the two classical cuspidal spaces thus
gives the full adelic dimension with this character.

We first consider the nonprincipal component, identified with
\(\mathfrak H^2\). Let \(M_k\) be its space of weight-\(k\)
forms invariant under the unit action. This space is the direct sum
of its subspaces symmetric and antisymmetric under Galois conjugation. By
\cite[Theorem~2.2]{Aoki2006Dimensions}, the symmetric subring is
the polynomial ring \(\mathbb C[G_2,G_3,G_4]\), where
\(G_j\) is an Eisenstein series of weight \(j\). The
antisymmetric part is a free module over this ring with one
generator of weight eleven. Adding the two Hilbert series gives
\[
\sum_{k\geq0}\dim M_kx^k
=\frac{1+x^{11}}{(1-x^2)(1-x^3)(1-x^4)}.
\]
In weights four and five, all these forms are symmetric.
Their bases are \(G_2^2,G_4\) and \(G_2\cdot G_3\), respectively,
so \(\dim M_4=2\) and \(\dim M_5=1\).

On the principal component, the Hilbert series for all
\(\mathrm{SL}_2(\mathcal O_F)\)-modular forms is
\((1+x^{10})/((1-x)(1-x^3)(1-x^4))\)
by \cite[\S4.1]{Aoki2006Dimensions}.
By \cite[Theorem~8.1]{Dursthoff2016}, the forms symmetric
under Galois conjugation form the polynomial ring
\(\mathbb C[g_1,s_3,g_4]\), with generators of weights one,
three and four. Its bases in weights four and five are
\(g_1^4,g_1\cdot s_3,g_4\) and \(g_1^5,g_1^2\cdot s_3,g_1\cdot g_4\).
Thus the symmetric subspace has dimension three in each weight,
equal to the dimension given by the full Hilbert series.
Every form in these two weights is therefore symmetric.
The subring fixed by the unit action is
\(\mathbb C[g_1,g_4,g_6]\), with generators of weights one,
four and six; see \cite[\S3.2]{Dursthoff2016}.
Its weight-four and weight-five bases are \(g_1^4,g_4\) and
\(g_1^5,g_1\cdot g_4\), respectively. Thus the modular-form dimension
on the principal component is two in both weights.

Each classical group has one cusp, since \(\mathcal O_F\) is a
principal ideal domain. The Eisenstein series \(G_2,G_3\) and
\(g_1\) have nonzero constant terms; see
\cite[\S3.2]{Dursthoff2016}.
Thus the constant-term map has rank one in both weights, using
\(G_2^2,G_2\cdot G_3\) on the nonprincipal component and
\(g_1^4,g_1^5\) on the principal component.
The cuspidal dimension is therefore one less than the modular-form
dimension on each component, giving\footnote{
We verify the Hilbert-series coefficients in weights four
and five, obtaining $\dim S_4(1)=2$ and
$\dim S_5(\varepsilon)=1$; see \path{verify_d12_weight_one.py}
in \cite{EigenformCertificates}.}
\[
\dim S_4(1)=2,\qquad
\dim S_5(\varepsilon)=1.
\]
\end{proof}

We now compute the constant term of $E_1(1,\varepsilon)$ and
use the one-dimensional space $S_5(\varepsilon)$ to obtain an
identity. We also show that the product is the same for both
normalized weight-four eigenforms.

\begin{proposition}
\label{prop:weight-one-sqrt3}
Let \(F=\mathbb Q(\sqrt3)\), and let \(\varepsilon\) be its nontrivial
narrow class character. For every normalized full-level cuspidal Hecke eigenform
\(h_4\in S_4(1)\), the following identity holds:
\[
h_5=12E_1(1,\varepsilon)\cdot h_4,
\]
where \(h_5\) is the unique normalized Hecke eigenform in
\(S_5(\varepsilon)\). In particular, the two normalized weight-four eigenforms
give the same product.
\end{proposition}

\begin{proof}
The Artin factorization in the proof of
Lemma~\ref{lem:quadratic-exact-special-values} gives
\(L(s,\varepsilon)=L(s,\chi_{-4})L(s,\chi_{-3})\), where
\(\chi_d\) denotes the primitive quadratic Dirichlet character of
discriminant \(d\). Applying
\eqref{eq:generalized-bernoulli-special-value} at \(m=1\), we obtain
\(L(0,\chi_{-4})=1/2\) and \(L(0,\chi_{-3})=1/3\).\footnote{
We compute $L(0,\chi_{-4})=1/2$ and $L(0,\chi_{-3})=1/3$;
see \path{verify_d12_weight_one.py} in \cite{EigenformCertificates}.}
Thus \(L(0,\varepsilon)=1/6\), and
\eqref{eq:weight-one-constant-vector} yields
\[
c_1(0,E_1(1,\varepsilon))
=\frac{L(0,\varepsilon)+L(0,\varepsilon^{-1})}{4}
=\frac1{12}.
\]

Put $E=E_1(1,\varepsilon)$. The same formula gives
the constant term on component $\nu$:
\[
c_\nu(0,E)=\frac{1+\varepsilon(t_\nu)}{24}.
\]
For an ideal $\mathfrak a$ in the nonprincipal narrow class,
$\varepsilon(\mathfrak a)=-1$. The nonconstant coefficient formula
in \eqref{eq:eisenstein-fourier-and-constant-terms}, which also
holds in weight one, gives
\[
c(\mathfrak a,E)=\sum_{\mathfrak r\mid\mathfrak a}
\varepsilon(\mathfrak r)=0.
\]
Indeed, the involution
$\mathfrak r\mapsto\mathfrak a\mathfrak r^{-1}$ pairs
opposite character values and has no fixed point, since
$\varepsilon(\mathfrak r^2)=1$.
The Fourier coefficients on this component correspond to ideals
in the nonprincipal narrow class. Thus $E$ vanishes on the
nonprincipal component, whereas its restriction to the principal
component is nonzero because its constant term is $1/12$.

Now let $h_4\in S_4(1)$ be a normalized Hecke eigenform.
Its restriction to the principal component is nonzero, since
$c(\mathcal O_F,h_4)=1$. The weight-four cuspidal space on
this component is one-dimensional by
the proof of Lemma~\ref{lem:weight-one-sqrt3-dimensions}.
The two normalized eigenforms therefore have identical
restrictions to it. Multiplication by $E$ annihilates their other
component, so their products agree.

The unit-ideal coefficient of $E\cdot h_4$ is $1/12$ by
\eqref{eq:product-coefficient-one}. The product belongs to the
one-dimensional Hecke-stable space $S_5(\varepsilon)$, so
$12E\cdot h_4$ is its normalized Hecke eigenform $h_5$.
\end{proof}

\begin{proof}[Proof of Theorem~\ref{thm:weight-one-identity}]
Lemma~\ref{lem:weight-one-sqrt3-dimensions} and
Proposition~\ref{prop:weight-one-sqrt3} give the assertion.
\end{proof}
\medskip
\noindent\textbf{Declaration of AI use.}
During the preparation of this manuscript, the authors used ChatGPT to
assist with code development. The authors ran and checked all code and
take responsibility for the content of the manuscript.

\end{document}